\documentclass{amsart}
\usepackage{amssymb,latexsym,amsmath,amsfonts,amsthm,amscd,graphicx,url,color,hyperref,tikz}

\theoremstyle{plain}
\newtheorem*{mainthm}{Main Theorem}
\newtheorem{thm}{Theorem}[section]
\newtheorem{prop}[thm]{Proposition}
\newtheorem{cor}[thm]{Corollary}
\newtheorem{lemma}[thm]{Lemma}

\theoremstyle{definition}
\newtheorem{definition}[thm]{Definition}
\newtheorem{remark}[thm]{Remark}

\newtheorem{examples}[thm]{Examples}
\newtheorem{example}[thm]{Example}
\newtheorem{notation}[thm]{Notation}

\newcommand{\BiProj}{\mathrm{BiProj}}
\newcommand{\Tor}{\mathrm{Tor}}
\newcommand{\HS}{\mathrm{HS}}

\newcommand{\RI}{\mathcal{R}(I)}
\newcommand{\FIf}{\mathbb{K}[\mathbf{f}]}
\newcommand{\FIg}{\mathbb{K}[\mathbf{g}]}

\newcommand{\SR}{S_\mathcal{R}}
\newcommand{\SF}{S_\mathcal{F}}
\newcommand{\PiR}{\Pi_\mathcal{R}}
\newcommand{\PiF}{\Pi_\mathcal{F}}
\newcommand{\mI}{\mathcal{I}_\Delta}
\newcommand{\mJ}{\mathcal{J}}
\newcommand{\mK}{\mathcal{K}}
\newcommand{\mL}{\mathcal{L}}
\newcommand{\LL}{L_{\alpha,\beta,\gamma}}
\newcommand{\MM}{M_{\alpha,\beta,\gamma}}
\newcommand{\PP}{P_{\alpha,\beta,\gamma,\delta}}
\newcommand{\QQ}{Q_{\alpha,\beta,\gamma,\delta}}

\newcommand{\LM}{\mathrm{LM}}
\newcommand{\iin}{\mathrm{in}}
\newcommand{\mdeg}{\mathrm{mdeg}}
\newcommand{\sdeg}{\mathrm{sdeg}}
\newcommand{\ee}{\mathbf{e}}
\renewcommand{\S}{\mathrm{S}}
\newcommand{\ch}{\mathrm{char}}
\newcommand{\lcm}{\mathrm{lcm}}

\newcommand{\mat}{\mathbf{M}}
\newcommand{\matX}{\mathbf{X}}
\newcommand{\al}{\alpha}
\newcommand{\be}{\beta}
\newcommand{\ga}{\gamma}
\newcommand{\de}{\delta}
\newcommand{\ep}{\varepsilon}
\newcommand{\oa}{\overline{\alpha}}
\newcommand{\ob}{\overline{\beta}}
\newcommand{\og}{\overline{\gamma}}
\newcommand{\od}{\overline{\delta}}
\renewcommand{\oe}{\overline{\varepsilon}}

\newcommand{\Card}{\mathrm{Card}}
\newcommand{\length}{\mathrm{length}}
\newcommand{\fI}{\mathfrak{I}}
\newcommand{\fw}{\mathfrak{w}}
\newcommand{\fb}{\mathfrak{b}}

\newcommand{\ba}{\mathbf{a}}
\newcommand{\bb}{\mathbf{b}}

\newcommand{\bff}{\mathbf{f}}
\newcommand{\bg}{\mathbf{g}}

\newcommand{\bi}{\mathbf{i}}
\newcommand{\bj}{\mathbf{j}}
\newcommand{\bh}{\mathbf{h}}

\newcommand{\Supp}{\mathrm{Supp}}

\makeatletter
\newcommand{\xdashrightarrow}[2][]{\ext@arrow 0359\rightarrowfill@@{#1}{#2}}
\def\rightarrowfill@@{\arrowfill@@\relax\relbar\rightarrow}
\def\arrowfill@@#1#2#3#4{%
  $\m@th\thickmuskip0mu\medmuskip\thickmuskip\thinmuskip\thickmuskip
   \relax#4#1
   \xleaders\hbox{$#4#2$}\hfill
   #3$%
}
\makeatother

\title{Blowup algebras of rational normal scrolls}
\author{Alessio Sammartano}
\address{Department of Mathematics, University of Notre Dame, 255 Hurley, Notre Dame, IN 46556, USA}
\email{asammart@nd.edu}


\begin{document}

\begin{abstract}
We determine the  equations of the blowup of $\mathbb{P}^{n}$
along a $d$-fold rational normal scroll  $\mathcal{S}$,
and we prove that the Rees ring and special fiber ring of $\mathcal{S}\subseteq\mathbb{P}^{n}$ are  Koszul algebras.
\end{abstract}

\maketitle

\section*{Introduction}

Let $R=\mathbb{K}[x_0,\ldots,x_n]$ be a polynomial ring over a field, 
$\bff=\{f_0, ..., f_s\}\subseteq R$  a set of forms of the same degree,
 and  $I = (\bff)\subseteq R$ the ideal generated by them.
 We consider the algebras
$
\RI =  R[f_0t, \ldots, f_st]\subseteq R[t],
$
known as the   \emph{Rees ring} of $I$,
and $\mathbb{K}[\bff] = \mathbb{K}[f_0, \ldots, f_s] \subseteq R$,
known as the \emph{special fiber ring} of $I$.
They are also known as \emph{blowup algebras} of $I$, 
since $\BiProj (\RI)$ is the  blowup of $\mathbb{P}^n$ along the subscheme defined by $I$.
Blowup algebras are central objects in the study of  rational maps \cite{DHS,KPU,RS,S},
  Castelnuovo-Mumford regularity   \cite{BCV,Ch,CHT,R}, geometric modeling \cite{Bu,C},
  and several other areas such as 
  linkage, residual intersections, multiplicity theory,
  integral dependence,
singularities, and  combinatorics.

A classical problem in algebra and geometry is the determination of the defining relations of  $\FIf$ and $\RI$.
It amounts to finding the implicit equations of graphs and images of rational maps in projective space,
and is therefore called the \emph{implicitization} \emph{problem}.
Unfortunately, this problem is known to be quite hard, 
and has been solved in a very short list of cases.
In the well known case of maximal minors of a generic matrix,
$\FIf$ is   the coordinate ring of a Grassmann variety and  is defined by the Pl\"ucker relations.
However, the relations among non-maximal minors  are still unknown \cite{BCV2}.
For symmetric matrices, 
a set of generators for the defining ideal of $\FIf$ when $\bff$ consists of all principal minors  is conjectured in \cite{HS},
 and established  in \cite{O} only up to radical.
The problem for Rees rings $\RI$ is even harder, and remains unsolved in all but a handful of cases, including
linearly presented ideals $I$ that are Cohen-Macaulay of codimension two \cite{MU} or Gorenstein of codimension three \cite{KPU2}.
It is even open for three-generated ideals  $I\subseteq\mathbb{K}[x_0,x_1]$ \cite{Bu}.

We focus on rational normal scrolls,
which are prominent  algebraic varieties, both in classical results and more recent developments \cite{BSV,CJ,EH,KPU1}.
Let $n_1, \ldots, n_d$ be positive integers and $c = \sum n_i$.
Choose rational normal curves $\mathcal{C}_i \subseteq \mathbb{P}^{c+d-1}$ of degree $n_i$ with complementary linear spans,
 and  isomorphisms $\varphi_i : \mathbb{P}^1 \rightarrow \mathcal{C}_i$ for each $i=1,\ldots, d$. 
The corresponding rational normal scroll  is
$$
\mathcal{S}_{n_1, \ldots, n_d} = \bigcup_{p \in \mathbb{P}^1} \overline{\varphi_1(p), \varphi_2(p), \ldots, \varphi_d(p)} \subseteq \mathbb{P}^{c+d-1}
$$
and is uniquely determined by $n_1, \ldots, n_d$ up to projective equivalence.
In suitable coordinates,
the  ideal of  $\mathcal{S}_{n_1, \ldots, n_d}\subseteq \mathbb{P}^{c+d-1}$ is generated by the minors of the matrix
$$
\left(
\begin{matrix}
x_{1,0}\! &\! x_{1,1}  \!&\! \cdots \!&\! x_{1,n_1-1} \\
x_{1,1} \!&\! x_{1,2} \! &\! \cdots \!&\! x_{1,n_1}
\end{matrix}
\right|
\left.
\begin{matrix}
x_{2,0}\! &\! x_{2,1}  \!&\! \cdots \!&\! x_{2,n_2-1} \\
x_{2,1} \!& \! x_{2,2} \! &\! \cdots \!&\! x_{2,n_2}
\end{matrix}
\right|
\left.
\begin{matrix}
 \cdots \, \\
  \cdots \,
\end{matrix}
\right|
\left.
\begin{matrix}
x_{d,0}\! & \!x_{d,1} \! &\! \cdots \!&\! x_{d,n_d-1} \\
x_{d,1} \!&\! x_{d,2}\!  & \!\cdots\! &\! x_{d,n_d}
\end{matrix}
\right).
$$

The implicitization problem 
was solved by Conca, Herzog, and Valla \cite{CHV} under the  assumption that the scroll is  \emph{balanced},
 i.e.  $| n_i - n_j | \leq 1$ for all $i,j$.
The authors construct degenerations  of $\RI$ and $\FIf$ to toric rings,
and obtain the  equations of the blowup algebras by lifting  the toric ideals.
They deduce that $\RI$ and $\FIf$ are Koszul algebras,
and by  \cite{B} it follows that  $I$ has linear powers, that is, the Castelnuovo-Mumford regularity of the powers of $I$ is the  least  possible. 
It has been an open question for a long time whether these strong results hold for all rational normal scrolls, 
in particular whether all scrolls have linear powers and whether there exists a Gr\"obner basis  of quadrics for the ideal of relations among the minors.
Recently, 
Bruns, Conca, and Varbaro \cite{BCV} proved  that all rational normal scrolls have linear powers.
Their method  is different from the one of \cite{CHV} though, and  the  equations of the blowup  remained unknown.

In this paper we solve the implicitization problem for rational normal scrolls in full generality, cf. the Main Theorem in Section \ref{SectionRelations}.
We determine the  equations of the blowup of $\mathbb{P}^{c+d-1}$ along  $\mathcal{S}_{n_1,\ldots,n_d}$ and of the special fiber 
in their natural embeddings in  $\mathbb{P}^{c+d-1} \times  \mathbb{P}^{{c-1 \choose 2}-1}$ and $ \mathbb{P}^{{c-1 \choose 2}-1}$ respectively.
We exhibit a new class of quadratic relations,
  arising from the vanishing of certain  determinants,
   which,
together with  the classical Pl\"ucker relations and the  relations of the symmetric algebra of $I$, suffice to generate the ideals of relations.
As a consequence, the blowup of $\mathbb{P}^n$ along any  variety of minimal degree is cut out  by quadric hypersurfaces.

In order to prove that these quadratic relations generate all the relations of the blowup algebras,
we introduce in Section \ref{SectionTermOrder} a new  construction of term orders.
In fact, 
one of the technical reasons why this problem has remained elusive in the last two decades is the failure of 
 traditional Gr\"obner  techniques, 
including those inspired by the theory of algebras with straightening laws.
Our term order  allows to construct   a well-behaved  subcomplex of the non-crossing complex
whose combinatorial patterns are governed by the Catalan trapezoids,
and which yields the expected Hilbert function for  the quadratic relations of $\FIf$,
cf. Section \ref{SectionFiber}.
Exploiting the results on $\FIf$ and the fiber type property of  $I$ \cite{BCV},
in Section \ref{SectionRees} 
we deduce the defining equations of $\RI$ and  show  that they form a squarefree quadratic Gr\"obner basis.
We adopt a mixed strategy to prove this result:
we bypass the explicit computation of the majority of S-pairs by reducing to a finite set of Hilbert series equations, which are verified with a computer-assisted calculation,
whereas we exploit some   syzygies of the defining ideal of $\RI$ to prove that the remaining S-pairs reduce to 0.
In particular, we prove that $\RI$ and $\FIf$ are Koszul algebras.

\section{Relations among minors}\label{SectionRelations}

In this section we use the standard presentation of the ideal of a rational normal scroll to  describe the equations of its blowup algebras  in detail, and state our main results.
We refer to \cite{BV,MS} for generalities on determinantal ideals.

From now on we assume, without loss of generality, that $n_1\leq n_2 \leq \cdots \leq n_d$. 
For each $i=1,\ldots,d$ let $X_i = \{x_{i,0}, \ldots, x_{i,n_i}\}$ be disjoint sets of variables and $R = \mathbb{K}[x_{i,j}]$ 
 a polynomial ring in the union of the $ X_i$ over an arbitrary field $\mathbb{K}$.
 Consider the matrix 
$$
\matX =
\left(
\begin{matrix}
x_{1,0} \! & \! x_{1,1}\!  & \!\cdots \!& \! x_{1,n_1-1} \\
x_{1,1} \! & \! x_{1,2}  \!&\! \cdots \!&\! x_{1,n_1}
\end{matrix}
\right|
\left.
\begin{matrix}
x_{2,0}\! & \! x_{2,1}\!  &\! \cdots \!&\! x_{2,n_2-1} \\
x_{2,1} \! &\!  x_{2,2}  \!& \!\cdots \!& \! x_{2,n_2}
\end{matrix}
\right|
\left.
\begin{matrix}
\, \cdots \, \\
 \, \cdots \,
\end{matrix}
\right|
\left.
\begin{matrix}
x_{d,0} \!&\! x_{d,1} \! &\! \cdots \!& \!x_{d,n_d-1} \\
x_{d,1} \!&\! x_{d,2} \! & \!\cdots \!&\! x_{d,n_d}
\end{matrix}
\right)
$$
and denote by $\xi_{i,j}$ its $(i,j)$-entry.
The  homogeneous ideal of  $\mathcal{S}_{n_1,\ldots,n_d}$ is $I= (\bff)\subseteq R$ where 
$\bff = \{f_{\al,\be}\,|\, 1 \leq \al<\be \leq c\}$, $f_{\al,\be}=\xi_{1,\al}\xi_{2,\be}-\xi_{1,\be}\xi_{2,\al}$, and $c = \sum n_i$.

Let $t$ be a new variable.
The Rees ring  $\RI =  R[It]\subseteq R[t]$ is bigraded  by letting $\deg x_{i,j} = (1,0)$ and $\deg t = (-2,1)$.
In this way $\RI$ is a standard bigraded $\mathbb{K}$-algebra in the sense that it is generated in bidegrees $(1,0)$ and $(0,1)$.
The special fiber ring $ \RI \otimes_R \mathbb{K} \cong \FIf  \subseteq R$ can be identified with the subring of $\RI$ concentrated in bidegrees $(0,\ast)$, 
and is a standard graded $\mathbb{K}$-algebra.
Introduce new  variables  $\{Y_{\al,\be} \, |\, 1\leq \al < \be \leq c\}$
and define a bihomogenoeus presentation
$$
\Psi_\mathcal{R} :  R[Y_{\al,\be}] \twoheadrightarrow \RI,  \qquad Y_{\al,\be}\mapsto tf_{\al,\be},\quad x_{i,j}\mapsto x_{i,j},
$$
and the induced homogeneous presentation
$$
\Psi_\mathcal{F} : \mathbb{K}[Y_{\al,\be}] \twoheadrightarrow \FIf, \qquad Y_{\al,\be}\mapsto f_{\al,\be}.
$$
The ideal $\ker \Psi_\mathcal{R}$ is bigraded, while $\ker \Psi_\mathcal{F}$ is the graded ideal generated by the elements of $\ker \Psi_\mathcal{R}$ of bidegree $(0,\ast)$, in other words
$\ker \Psi_\mathcal{F}=\ker \Psi_\mathcal{R} \cap \mathbb{K}[Y_{\al,\be}]$.

\begin{notation}
For a positive integer $n$ we set $[n]=\{1,\ldots, n\}$.
Given $\Gamma \subseteq [c]$ and  $k\in \mathbb{N}$
we denote by ${ \Gamma \choose k}$ the set of strictly increasing $k$-tuples  of elements of $\Gamma$.
\end{notation}

Some relations of the Rees ring derive from the syzygies of $I$.
The resolution of the ideal of maximal minors of the matrix $\matX$ is an Eagon-Nortchott complex.
In our  case, we can describe the first syzygies as follows:
for any $(\al,\be,\ga)\in {[c] \choose 3}$ the following determinants vanish
$$
\left|
\begin{matrix}
\xi_{1,\al} & \xi_{1,\be} & \xi_{1,\ga}  \\
\xi_{1,\al} & \xi_{1,\be} & \xi_{1,\ga}  \\
\xi_{2,\al} & \xi_{2,\be} & \xi_{2,\ga}  
\end{matrix}
\right|
=
\left|
\begin{matrix}
\xi_{1,\al} & \xi_{1,\be} & \xi_{1,\ga}  \\
\xi_{2,\al} & \xi_{2,\be} & \xi_{2,\ga}  \\
\xi_{2,\al} & \xi_{2,\be} & \xi_{2,\ga}  
\end{matrix}
\right|
= 0.
$$
Expanding the  determinants along  the first and third row respectively  gives rise to 
\begin{equation}\label{LinearRelationsX}
\xi_{1,\al}Y_{\be,\ga}-\xi_{1,\be}Y_{\al,\ga}+\xi_{1,\ga}Y_{\al,\be}, \,\,
\xi_{2,\al}Y_{\be,\ga}-\xi_{2,\be}Y_{\al,\ga}+\xi_{2,\ga}Y_{\al,\be} \in \ker \Psi_\mathcal{R}.
\end{equation}

The Pl\"ucker relations are 
a well-known  set of relations among maximal minors.  
For a $2\times c$ matrix they take the form 
\begin{equation}\label{PluckerRelationsX}
Y_{\al,\be}Y_{\ga,\de} - Y_{\al,\ga}Y_{\be,\de} + Y_{\al,\de}Y_{\be,\ga}\in \ker \Psi_\mathcal{R}, \qquad (\al,\be,\ga,\de)\in {[c] \choose 4}.
\end{equation}

Next, we describe a new set of relations among the minors of $\matX$.
For any four columns avoiding the terminal column of each catalecticant block, 
that is, for any
$(\al,\be,\ga,\de)\in { \Gamma\choose 4}$ with $\Gamma = [c]\setminus  \{n_1, n_1+n_2, \ldots, \sum n_i\}$,
the following  determinant vanishes 
$$
\left|
\begin{matrix}
\xi_{1,\al} & \xi_{1,\be} & \xi_{1,\ga} & \xi_{1,\de} \\
\xi_{2,\al} & \xi_{2,\be} & \xi_{2,\ga} & \xi_{2,\de} \\
\xi_{1,\al+1} & \xi_{1,\be+1} & \xi_{1,\ga+1} & \xi_{1,\de+1} \\
\xi_{2,\al+1} & \xi_{2,\be+1} & \xi_{2,\ga+1} & \xi_{2,\de+1}
\end{matrix}
\right|
= 0.
$$
Expanding the determinant along the first two rows
we obtain the quadratic relation 
\begin{equation}\label{NonPluckerRelationsX}
Y_{\al,\be}Y_{\ga+1,\de+1} - Y_{\al,\ga}Y_{\be+1,\de+1} + Y_{\al,\de}Y_{\be+1,\ga+1}
\end{equation}
\begin{equation*}
+Y_{\be,\ga}Y_{\al+1,\de+1} - Y_{\be,\de}Y_{\al+1,\ga+1} + Y_{\ga,\de}Y_{\al+1,\be+1} \in \ker \Psi_\mathcal{R}.
\end{equation*}

We will prove that these polynomials suffice to generate  $\ker \Psi_\mathcal{R}$ and $\ker \Psi_\mathcal{F}$.

\begin{mainthm}\label{MainTheorem}
The defining ideal 
of $\RI$ is minimally generated by the polynomials \eqref{LinearRelationsX}, \eqref{PluckerRelationsX}, and \eqref{NonPluckerRelationsX},
whereas
the defining ideal  
of $\FIf$ is minimally generated by the polynomials \eqref{PluckerRelationsX} and \eqref{NonPluckerRelationsX}.
Moreover, these generating sets are  Gr\"obner bases with respect to  suitable term orders.
\end{mainthm}

\section{Term order}\label{SectionTermOrder}

In this section we introduce a more involved presentation of  the ideal $I$ and use it to construct a term order for the defining ideal of $\RI$.
Both the new presentation and the term order play important roles in the proofs of the main results in Sections \ref{SectionFiber} and \ref{SectionRees}.
To the best of our knowledge,  this presentation of the rational normal scrolls and the term order on the algebraic relations among the minors have not been considered before, and may be of use in  other problems on determinantal varieties.

We consider another  matrix $\mat$ obtained by rearranging the columns of  $\matX$.
Let  $i(\ell)$ denote the least integer  $i$ such that $n_i\geq \ell$.
The first $c-d$ columns of $\mat$ are
$$ 
\left.
\begin{matrix}
x_{i(2),0}\! &\! x_{i(2)+1,0}\!  &\! \cdots \!&\! x_{d,0} \!&\! x_{i(3),1} \!&\! x_{i(3)+1,1}  \!&\! \cdots \!&\! x_{d,1} \!&\! x_{i(4),2} \!&\! \cdots \!&\! \cdots \!&\! x_{d,n_{d}-2}\\
x_{i(2),1}\! &\! x_{i(2)+1,1}\!  &\! \cdots \!&\! x_{d,1} \!&\! x_{i(3),2} \!&\! x_{i(3)+1,2} \!&\! \cdots \!&\! x_{d,2} \!&\! x_{i(4),3} \!&\! \cdots \!&\! \cdots \!&\! x_{d,n_{d}-1}
\end{matrix}
\right..
$$
In other words, 
we start with the first column of the $i$-th catalecticant block of $\matX$ for each $i$ increasingly in  $i$, 
then the second column, and so on until we have used all columns except the last one for each block;
when a block runs out of columns  we simply skip it. 
The last $d$ columns of $\mat$ are
$$ 
\left.
\begin{matrix}
x_{d,n_d-1} & x_{d-1,n_{d-1}-1}  & \cdots & x_{2,n_2-1} & x_{1,n_1-1}\\
x_{d,n_d} & x_{d-1,n_{d-1}}  & \cdots & x_{2,n_2} & x_{1,n_1} 
\end{matrix}
\right.
$$
i.e.,
they consist of the last column of the $i$-th block of $\matX$ for each $i$, 
but this time ordered decreasingly in  $i$.

\begin{examples}
We illustrate the construction of the matrix $\mat$ with several examples. Let $\mathbf{n}=(n_1, \ldots, n_d)$.
\begin{align*}
\!\mathbf{n}\!&=\!(4,4) \!
&&
\!\mat \!=\!
\left(
\begin{matrix}
x_{1,0}\! &\! x_{2,0} \! &\! x_{1,1}\! &\! x_{2,1} \! &\! x_{1,2} \!&\! x_{2,2}  \!&\! x_{2,3} \!& \! x_{1,3}   \\
x_{1,1}\! &\! x_{2,1}  \!&\! x_{1,2} \!&\! x_{2,2}  \!&\! x_{1,3} \!&\! x_{2,3}  \!&\! x_{2,4} \!&\!  x_{1,4}
\end{matrix}
\right),
\\
\!\mathbf{n}\!&=\!(3,3,4) \!
&&
\!\mat \!=\!
\left(
\begin{matrix}
x_{1,0}\! &\! x_{2,0} \! &\! x_{3,0} \!&\! x_{1,1} \! &\! x_{2,1} \!&\! x_{3,1}  \!&\! x_{3,2} \!&\! x_{3,3} \! &\! x_{2,2}\! &\! x_{1,2}   \\
x_{1,1}\! &\! x_{2,1}  \!&\! x_{3,1} \!&\! x_{1,2}  \!&\! x_{2,2} \!&\! x_{3,2}  \!&\! x_{3,3} \!&\! x_{3,4} \! &\! x_{2,3}\! &\! x_{1,3}
\end{matrix}
\right),
\\
\!\mathbf{n}\!&=\! (1,1,2,3,4) \!
&&
\!\mat \!= \! \left(
\begin{matrix}
x_{3,0}\! &\! x_{4,0}  \!&\! x_{5,0}  \! &\! x_{4,1} \!&\! x_{5,1} \! &\! x_{5,2} \!&\! x_{5,3} \! &\! x_{4,2} \!&\! x_{3,1} \!&\! x_{2,0} \!&\! x_{1,0}   \\
x_{3,1}\! &\! x_{4,1} \! &\! x_{5,1}   \!&\! x_{4,2} \!&\! x_{5,2} \! &\! x_{5,3} \!&\! x_{5,4} \! &\! x_{4,3} \!&\! x_{3,2} \!&\! x_{2,1} \!&\! x_{1,1}
\end{matrix}
\right),
\\
\!\mathbf{n}\!&=\! (1,1,\ldots,1) \!
&&
\!\mat\! =\!\left(
\begin{matrix}
x_{d,0} \!&\! x_{d-1,0} \!&\!x_{d-2,0} \!&\! \cdots \!&\! x_{2,0} \!&\! x_{1,0} \\
x_{d,1} \!&\! x_{d-1,1} \!&\!x_{d-2,1} \!&\! \cdots \!&\! x_{2,1} \!&\! x_{1,1}
\end{matrix}
\right),
\\
\!\mathbf{n}\!&=\! (n_1) \!
&&
\!\mat \!=\!
\left(
\begin{matrix}
x_{1,0} \!&\! x_{1,1} \! &\! x_{1,2} \! & \!\cdots \!&\! x_{1,n_1-1} \\
x_{1,1} \!&\! x_{1,2} \! &\! x_{1,3} \! &\! \cdots \!&\! x_{1,n_1} 
\end{matrix}
\right).
\end{align*}
\end{examples}

Denote by $\mu_{i,j}$ the $(i,j)$-entry of $\mat$ and by 
$g_{\al,\be} = \mu_{1,\al}\mu_{2,\be}-\mu_{1,\be}\mu_{2,\al}$ its $2\times 2$ minors, with $\al<\be$.
Thus  $I=(\bg)\subseteq R$ where 
$\bg = \big\{g_{\al,\be} \, |\, (\al,\be)\in {[c] \choose 2}\big\}$.
We let $\big\{T_{\al,\be} \, |\,  (\al,\be)\in {[c] \choose 2}\big\}$ be new variables and
define  (bi)homogeneous presentations for the Rees ring and the special fiber ring of $I$
\begin{align*}
\PiR &: \SR := R[T_{\al,\be}] \twoheadrightarrow \RI, &&T_{\al,\be}\mapsto tg_{\al,\be}, \qquad x_{i,j}\mapsto x_{i,j},\\
\PiF &: \SF := \mathbb{K}[T_{\al,\be}] \twoheadrightarrow \FIg,  &&T_{\al,\be}\mapsto g_{\al,\be}.
\end{align*}
Then $\SR$ is a standard bigraded polynomial ring by setting $\deg x_{i,j} = (1,0)$ and $\deg T_{\al,\be} = (0,1)$.
We denote the defining ideal of $\RI$ by $\mJ = \ker \PiR$ and that of $\FIg$ by $\mK = \ker \PiF$.
We also consider the defining ideal $\mL$ of the symmetric algebra of $I$, 
that is, the sub-ideal of $\mJ$ generated by the elements of bidegree $(\ast, 1)$.

\begin{notation}\label{NotationIndicesVariables}
We  adopt the  conventions $T_{\al,\be}:=-T_{\be,\al}$ if $\al>\be$, and $T_{\al,\al}:=0$.
\end{notation}

Since the matrices $\matX$  and $\mat$ differ  by a permutation,
the isomorphism $\Theta$ between the two presentation rings of $\RI$ 
such that 
 $\Psi_\mathcal{R}=\PiR \circ \Theta  $ is defined  
 by a relabeling of the variables with some multiplication by $-1$;
 likewise for $\FIg$.
Specifically, let $\tau$ be the permutation of $\{1,\ldots,c\}$ that we apply to the columns of $\matX$ to obtain $\mat$.
The isomorphism $\Theta$ is given by the correspondence
$ Y_{\al,\be}\mapsto  T_{\tau(\al),\tau(\be)}$
for all $\al<\be$;
note that there is a $-1$ whenever $\tau(\al)>\tau(\be)$. 

Now we translate the polynomial relations \eqref{LinearRelationsX}, \eqref{PluckerRelationsX}, \eqref{NonPluckerRelationsX} to this new presentation.
The isomorphism $\Theta$ preserves the form of  
the  syzygies and the Pl\"ucker equations,
i.e.
the ideal  $\mL$ is generated by the polynomials 
\begin{eqnarray}
\LL &:= \mu_{1,\al}T_{\be,\ga}-\mu_{1,\be}T_{\al,\ga}+\mu_{1,\ga}T_{\al,\be},\label{equationL}\\
\MM &:= \mu_{2,\al}T_{\be,\ga}-\mu_{2,\be}T_{\al,\ga}+\mu_{2,\ga}T_{\al,\be},\label{equationM}
\end{eqnarray}
for    $ (\al,\be,\ga)\in{[c] \choose 3}$,
while the Pl\"ucker relations are
\begin{equation}\label{equationP}
\PP :=
T_{\al,\be}T_{\ga,\de} - T_{\al,\ga}T_{\be,\de} + T_{\al,\de}T_{\be,\ga}
\end{equation}
for    $ (\al,\be,\ga,\de)\in{[c] \choose 4}$.
However, we need a notation for  relations \eqref{NonPluckerRelationsX}.

\begin{notation}\label{NotationNextIndex}
For $\al\leq c-d$
define $\oa := \tau(\tau^{-1}(\al)+1)$, that is, the next column in $\mat$ involving  the same set of variables $X_i$ as $\al$.
\end{notation}
\noindent
The relations \eqref{NonPluckerRelationsX}  correspond then to 
\begin{equation}\label{equationQ}
\QQ :=
T_{\al,\be}T_{\og,\od} - T_{\al,\ga}T_{\ob,\od} + T_{\al,\de}T_{\ob,\og}
+T_{\be,\ga}T_{\oa,\od} - T_{\be,\de}T_{\oa,\og} + T_{\ga,\de}T_{\oa,\ob}
\end{equation}
and the condition on the four indices becomes
$ (\al,\be,\ga,\de)\in{[c-d]\choose 4}$.

\begin{example}
Let $\mathbf{n}=(1,2,2,3)$. The corresponding matrix is
$$
\mat = \left(
\begin{matrix}
x_{2,0} & x_{3,0}  & x_{4,0}  &x_{4,1} & x_{4,2}  & x_{3,1} & x_{2,1} & x_{1,0}   \\
x_{2,1} & x_{3,1}  & x_{4,1} &x_{4,2} & x_{4,3} & x_{3,2} & x_{2,2} & x_{1,1}
\end{matrix}
\right).
$$
There is exactly one relation \eqref{equationQ}, namely 
\begin{align*}
Q_{1,2,3,4} &= 
T_{1,2}T_{\overline{3},\overline{4}} - T_{1,3}T_{\overline{2},\overline{4}} + T_{1,4}T_{\overline{2},\overline{3}}
+T_{2,3}T_{\overline{1},\overline{4}}- T_{2,4}T_{\overline{1},\overline{3}} + T_{3,4}T_{\overline{1},\overline{2}}\\
&=
T_{1,2}T_{4,5} + T_{1,3}T_{5,6} - T_{1,4}T_{4,6}
-T_{2,3}T_{5,7}+ T_{2,4}T_{4,7} - T_{3,4}T_{6,7}.
\end{align*}
\end{example}

\begin{remark}\label{RemarkProofMainTheorem}
Since the isomorphism $\Theta$ is defined by a relabeling of  variables and changes of sign, it allows transferring not just  relations but term orders too. 
We will prove the Main Theorem using the presentations $\PiR,\PiF$ and the equations \eqref{equationL}, \eqref{equationM}, \eqref{equationP}, \eqref{equationQ}.
\end{remark}

Next, we  define the term order on  the presentation rings $\SR$ and $\SF$.
The construction is obtained by refining a term order using ordered monoids,
a generalization of the refinement  by a non-negative weight $\omega$ in \cite{St}.

\begin{definition}\label{DefinitionTermOrder}
We establish the following total order on the variables of $\SR$:
\begin{itemize}

\item 
 $ T_{\al,\be}\succ T_{\ga,\de}$ if $\al<\ga$ or $\al=\ga, \be<\de $;

\item $x_{i,j}\succ x_{k,\ell}$ if $j< n_i, \ell<n_k$  and either $j<\ell$ or $j=\ell,i<k $;

\item $x_{i,j}\succ x_{k,n_k}$ if  $ j<n_i$ or $j = n_i, i>k$;

\item 
 $ x_{i,j}\succ T_{\al,\be}$ for all $i,j,\al,\be $.

\end{itemize}
Then, consider the induced lexicographic term order on $\SR$.
We  refine this term order using the $\mathbb{Z}$-grading on $\SR$ defined by
$
\sdeg(T_{\al,\be}) = c-\al, \, \sdeg(x_{i,j})=0
$
for all $\al,\be,i,j$.
We further refine by a multigrading $\mdeg(\cdot)$ defined as follows.
Consider the semigroup $(\mathbb{N}^{c+d},+)$ with canonical basis $\ee_1 > \cdots > \ee_{c+d}$ and totally ordered by the lexicographic order. 
Set
$$
\mdeg(T_{\al,\be}) = \ee_{\al+d} + \ee_{\be+d}, \quad
\mdeg(\mu_{2,\ga}) = \ee_{\ga+d}, \quad
\mdeg(x_{i,0}) = \ee_{i} 
$$
The resulting term order is our desired $\prec$. 
On the subring $\SF\subseteq \SR$ we consider the restriction of this term order, which we also denote by $\prec$.
\end{definition}

\begin{remark}\label{RemarkFirstRow}
The total order on  the  variables $x_{i,j}$ is obtained by starting with  $x_{1,0}\succ x_{2,0}\succ \cdots \succ x_{d,0}$ and then appending the remaining variables with the order they appear on the second row of $\mat$. 
On the other hand, on the first row of $\mat$
we have
$
\mu_{1,1}\succ \mu_{1,2}\succ \cdots \succ \mu_{1,c-d} \succ \mu_{1,c-d+1} \prec  \mu_{1,c-d+2} \prec \cdots \prec \mu_{1,c}.
$
\end{remark}

\begin{remark}\label{RemarkOrderVariablesMultidegree}
For any $\al, \be \in [c]$ we have $\mu_{1,\al} \succ \mu_{1,\be}$ if and only if $\mdeg(\mu_{1,\al}) > \mdeg(\mu_{1,\be})$ in the lexicographic order on $\mathbb{N}^{c+d}$.
\end{remark}

Since $\prec$ is the only term order we consider in this paper, we denote leading monomials simply by $\LM(\cdot)$.

\begin{prop}\label{PropLeadingTerms}
The polynomials \eqref{equationM}, \eqref{equationP}, \eqref{equationQ} have leading monomials
$$
\LM(\PP) = T_{\al,\ga}T_{\be,\de}, 
\quad 
\LM(\QQ) = T_{\al,\be}T_{\og,\od},
\quad 
\LM(\MM) = \mu_{2,\be}T_{\al,\ga}.
$$
The leading monomial $ \LM(\LL)$  of \eqref{equationL} is the monomial
 containing the  $x_{i,j}$ with lowest $j$ and then lowest $i$.
\end{prop}

\begin{proof}
By Definition \ref{DefinitionTermOrder},
for each polynomial we must start by considering $\mdeg(\cdot)$.
Since $\PP$ and $\MM$ are homogeneous with respect to $\mdeg(\cdot)$,  the refinement has no effect on them.
On the other hand, $\QQ$ and $\LL$ are not homogeneous.
The unique monomial of highest multidegree  in the support of $\QQ$ is  $T_{\al,\be}T_{\og,\od}$ and hence  
we can already conclude that it is the leading monomial.

For $\LL$ we distinguish two cases.
If some  $x_{i,0}$ appears in $\LL$ then the unique monomial of highest multidegree contains the  $x_{i,0}$ with lowest $i$.
If all the $x_{i,j}$ appearing in $\LL$ have $j>0$
then 
$\al = \overline{\al_1},\be = \overline{\be_1},\ga = \overline{\ga_1}$
for some $\al_1,\be_1,\ga_1$;
it follows that the unique monomial of highest multidegree contains  $x_{i,j}=\mu_{2,\ep}$  where $\ep=\min\{\al_1,\be_1,\ga_1\}$, 
which is also the $x_{i,j}$ with lowest $j$ and then lowest $i$.
In either case, we deduce that $\LM(\LL)$ is the desired term.

The homogeneous component  of highest degree with respect to $\sdeg(\cdot)$ is
the binomial $
- T_{\al,\ga}T_{\be,\de} + T_{\al,\de}T_{\be,\ga}
$
for $\PP$ and
$
-\mu_{2,\be}T_{\al,\ga}+\mu_{2,\ga}T_{\al,\be}
 $
for $\MM$.
Finally, we break  ties using the lexicographic order:
we obtain
$
\LM(\PP) = T_{\al,\ga}T_{\be,\de}
$
and 
$
\LM(\MM) = \mu_{2,\be}T_{\al,\ga}.
$
\end{proof}

\begin{remark}\label{RemarkLeadingL}
From Remark \ref{RemarkFirstRow}  and Proposition \ref{PropLeadingTerms} it follows that $\LM(\LL)=\mu_{1,\al}T_{\be,\ga}$ or $\mu_{1,\ga}T_{\al,\be}$.
Furthermore, we can only have  $\LM(\LL)=\mu_{1,\ga}T_{\al,\be}$ when $\ga > c-d$.
\end{remark}

\section{The initial complex of the special fiber}\label{SectionFiber}

Let $\Delta$ denote the  flag simplicial complex whose Stanley-Reisner ideal is the ideal
$\mI\subseteq \SF = \mathbb{K}[T_{\al,\be}]$  generated by the quadratic squarefree monomials
\begin{align*}
\label{MonomialP}\tag{\dag}\LM(\PP)&=T_{\al,\ga} T_{\be,\de} &&(\al,\be, \ga,\de)\in {[c] \choose 4},\\
\label{MonomialQ}\tag{\ddag}\LM(\QQ)&=T_{\al,\be} T_{\og,\od} && (\al,\be, \ga,\de)\in {[c-d] \choose 4}.
\end{align*}
In  this section we will show that $\Delta$ is   an initial complex of the special fiber ring $\FIg$.
We refer to \cite{MS} for background on simplicial complexes.

The vertex set of $\Delta$ is 
$$ V = \big\{ (\al, \be)\, |\, 1\leq \al < \be \leq c\big\}. $$ 
We interpret its elements as open intervals in the real line $\mathbb{R}$ with integral endpoints,
and we will use the familiar notions of  length, intersection, subtraction, and containment of intervals.
The minimal non-faces of $\Delta$ are the pairs of vertices determined by \eqref{MonomialP} and \eqref{MonomialQ}.
The complex $\Delta$ exhibits different behavior in the two cases $c<d+4$ and $c\geq d+4$, so we treat them  separately.
In the former, the minors of $\mat$ parametrize the Grassmann variety of lines in $\mathbb{P}^{c-1}$,
cf. Remark \ref{RemarkInitialGrassmann}.

\subsection{Grassmann case: $c<d+4$.}
The set of monomials \eqref{MonomialQ} is empty, so the minimal non-faces of $\Delta$ correspond to the generators \eqref{MonomialP} of $\mI$.
In particular,  there is exactly one such complex $\Delta$ for each $c\geq 2$.
Moreover, the faces of $\Delta$ are the  subsets $F\subseteq V$  satisfying the  \emph{non-crossing} condition 
\begin{equation}\tag{$\diamondsuit$}\label{ConditionIntervals}
\text{for all }\fI_1,\fI_2\in F \text{ then }\fI_1\cap \fI_2 = \emptyset,\text{ or } \fI_1 \subseteq \fI_2, \text{ or }\fI_2\subseteq \fI_1.
\end{equation}

Denote by $C_n={2n \choose n}-{2n \choose n+1}$ the $n$-th Catalan number. 

\begin{prop}\label{PropositionDeltaGrassmann}
The complex $\Delta$ is pure  of dimension $2c-4$ and its number of facets is  $f_{2c-3}(\Delta) = C_{c-2}$. 
\end{prop}
\begin{proof}
Observe first that, by maximality, a facet $F$ of $\Delta$ is a face such that $(1,c)\in F$ and  
if $(\al,\be)\in F$ with $\be-\al>1$ then there exists a unique $\gamma$ with  $\al<\ga <\be$ and $(\al,\ga),(\ga,\be)\in F$.

We prove the proposition by induction on $c$, the case $c=2$ being trivial; assume $c\geq 3$.
Let $F$ be a facet and let  $1<\gamma<c$ be the unique integer with $(1,\ga),(\ga,c)\in F$.
Then we have 
$$
F= F_1 \sqcup F_2 \sqcup \{(1,c)\} \quad F_1=\{\fI \in F \, |\, \fI \subseteq(1,\ga)\}, \quad
 F_2=\{\fI \in F \, |\, \fI \subseteq(\ga,c)\}.
 $$
The sets $F_1$ and $F_2$ correspond to facets of the two smaller complexes $\Delta_{\gamma}$ and $\Delta_{c-\gamma+1}$ obtained when $\mat$ has respectively $\ga$ and $c-\ga+1$ columns.
By induction, $\Card(F_1) =2\gamma -3 $ and $\Card(F_2)= 2(c-\gamma+1)-3 $, thus $\Card(F) = 2c-3$ and $\Delta$ is pure with $\dim(\Delta) = 2c-4$.
Furthermore, 
the complexes $\Delta_{\gamma}$ and $\Delta_{c-\gamma+1}$ have respectively $C_{\gamma-2}$ and $C_{c-\gamma-1}$ facets.
Adding the contributions of each $\gamma = 2, \ldots, c-1$ and using  the well-known recursion of Catalan numbers $C_{n+1}= \sum_{i=0}^n C_iC_{n-i}$,
we conclude that the number of facets of $\Delta $ is 
$$
f_{2c-3}(\Delta) = \sum_{\gamma=2}^{c-1}C_{\gamma-2}C_{c-\gamma-1} = \sum_{\gamma'=0}^{c-3}C_{\gamma'}C_{c-3-\gamma'} = C_{c-2}. 
$$
\end{proof}

\begin{example}
Let $\mathbf{n}= (1,1,1,2,2,2)$, thus $c=9, d=6$.
The corresponding matrix is 
$$
\mat = \left(
\begin{matrix}
x_{4,0} &x_{5,0} & x_{6,0}  & x_{6,1}  &x_{5,1} & x_{4,1}  & x_{3,0} & x_{2,0} & x_{1,0}   \\
x_{4,1} & x_{5,1} & x_{6,1}  & x_{6,2} &x_{5,2} & x_{4,2} & x_{3,1} & x_{2,1} & x_{1,1}
\end{matrix}
\right).
$$
The collection of all the open intervals in the picture below is a facet of $\Delta$.
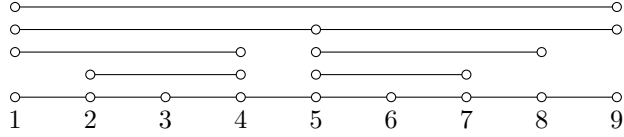
\vspace*{0.3cm}

\begin{center}
\begin{tikzpicture}\label{add}
\draw (0,0.3)--(8,0.3);
\draw (1,0.6)--(3,0.6);
\draw (4,0.6)--(6,0.6);
\draw (0,0.9)--(3,0.9);
\draw (4,0.9)--(7,0.9);
\draw (0,1.2)--(4,1.2);
\draw (4,1.2)--(8,1.2);
\draw (0,1.5)--(8,1.5);

\draw [fill=white] (0,0.3) circle [radius=0.06];
\draw [fill=white] (1,0.3) circle [radius=0.06];
\draw [fill=white] (2,0.3) circle [radius=0.06];
\draw [fill=white] (3,0.3) circle [radius=0.06];
\draw [fill=white] (4,0.3) circle [radius=0.06];
\draw [fill=white] (5,0.3) circle [radius=0.06];
\draw [fill=white] (6,0.3) circle [radius=0.06];
\draw [fill=white] (7,0.3) circle [radius=0.06];
\draw [fill=white] (8,0.3) circle [radius=0.06];

\draw [fill=white] (1,0.6) circle [radius=0.06];
\draw [fill=white] (3,0.6) circle [radius=0.06];
\draw [fill=white] (4,0.6) circle [radius=0.06];
\draw [fill=white] (6,0.6) circle [radius=0.06];

\draw [fill=white] (0,0.9) circle [radius=0.06];
\draw [fill=white] (3,0.9) circle [radius=0.06];
\draw [fill=white] (4,0.9) circle [radius=0.06];
\draw [fill=white] (7,0.9) circle [radius=0.06];
\draw [fill=white] (0,1.2) circle [radius=0.06];
\draw [fill=white] (0,1.5) circle [radius=0.06];
\draw [fill=white] (4,1.2) circle [radius=0.06];
\draw [fill=white] (8,1.2) circle [radius=0.06];
\draw [fill=white] (8,1.5) circle [radius=0.06];

\node at (0,0) {1};
\node at (1,0) {2};
\node at (2,0) {3};
\node at (3,0) {4};
\node at (4,0) {5};
\node at (5,0) {6};
\node at (6,0) {7};
\node at (7,0) {8};
\node at (8,0) {9};

\end{tikzpicture}
\end{center}
\end{example}

\subsection{Non Grassmann case: $c\geq d+4$.}
Both sets of monomials \eqref{MonomialP} and \eqref{MonomialQ} are nonempty. 
The faces of $\Delta$ still satisfy the non-crossing condition \eqref{ConditionIntervals}, but are also subject to new constraints arising from the minimal non-faces  \eqref{MonomialQ}.
We start by investigating these constraints. 

\begin{lemma}\label{LemmaNextIndicesNested}
Let $\al, \be, \ga, \de \in [c]$ be indices  such that $\al < \be, \ga$ and $\ob < \de < \og$.
Then there exists $\varepsilon \in [c]$ such that $\al < \varepsilon$ and $\de = \oe$.
\end{lemma}
\begin{proof}
Let $\mu_{2,\be}= x_{i_\be,j_\be},\mu_{2,\ga}=x_{i_\ga,j_\ga}, \mu_{2,\de}=x_{i_\de,j_\de}$,
then $\mu_{2,\ob}= x_{i_\be,j_\be+1}$ and $ \mu_{2,\og}=x_{i_\ga,j_\ga+1}$,
cf. the construction of $\mat$.

Suppose $\de \leq c-d$, i.e. that the $\de$-th column of $\mat$ is not the last of its block.
Since $\de>\ob$, we must have either $j_\de >  j_{\be}+1$ or $j_\de =  j_{\be}+1$ and $i_\de > i_\be$.
In either case, 
choosing $\varepsilon$ so that $\mu_{2,\varepsilon} = x_{i_\de, j_{\de}-1}$, we have $\varepsilon> \be $ and $\oe = \de$.

Suppose now $\de> c-d$.
Since $\de<\og$, we must have $i_\de >i_\ga$ and $j_{\ga}+1=n_{i_\ga}\leq n_{i_\de}=j_{\de}$.
Choosing $\varepsilon$ so that $\mu_{2,\varepsilon} = x_{i_\de, j_{\de}-1}$, we have $\varepsilon> \ga $ and $\oe = \de$.
\end{proof}

\begin{cor}\label{CorollaryReductionToMinimalIntervals}
Let $\{(\al,\be),(\ga,\de)\}$ be a minimal non-face of $\Delta$ of type \eqref{MonomialQ}.
For any sub-intervals $(\al_1,\be_1)\subseteq (\al,\be), (\ga_1,\de_1)\subseteq(\ga,\de)$,
the pair
 $\{(\al_1,\be_1),(\ga_1,\de_1)\}$ is also  a minimal non-face of $\Delta$ of type \eqref{MonomialQ}.
\end{cor}
\begin{proof}
By assumption there exist  $\beta <\varepsilon, \zeta$ such that $\overline{\varepsilon}= \ga, \overline{\zeta}= \de$.
Applying Lemma \ref{LemmaNextIndicesNested} twice, first to  $\be, \varepsilon, \zeta, \ga_1$ and then to $\be, \varepsilon, \zeta, \de_1$, 
we see that there exist  $\beta <\varepsilon_1, \zeta_1$
with $\overline{\varepsilon_1}= \ga_1, \overline{\zeta_1}= \de_1$.
It follows that $\{(\al_1,\be_1),(\ga_1,\de_1)\}$ is   a minimal non-face of $\Delta$ of type \eqref{MonomialQ}.
\end{proof}

Given a subset $F \subseteq V$, 
we say that an interval $(\al, \be)\in F$ is \emph{minimal} if it is a minimal element of the poset $(F, \subseteq)$.
Corollary \ref{CorollaryReductionToMinimalIntervals} says that
a non-crossing subset $F \subseteq V$ is a face of $\Delta$ if and only if $\{\fI_1, \fI_2\}$ is a face of $\Delta$ 
for any two  minimal intervals $\fI_1, \fI_2$ of $F$.
Motivated by this observation, we turn to the description of minimal intervals of facets of $\Delta$.

We say that an  interval $\fI\in V$ is  \emph{unitary} if $\fI=(\al, \al+1)$ for some $\al\in[c-1]$.
Notice that, unlike the Grassmann case,  a facet  does not contain all the unitary intervals of $V$,
cf. Example \ref{ExampleNonGrassmanFacetTree}.

\begin{lemma}\label{LemmaMinimalIntervalsAreUnitary}
Let $F$ be a facet of $\Delta$.
The minimal intervals of $F$  with respect to inclusion are unitary.  
\end{lemma}
\begin{proof} 
Let $(\al,\be)$ be a minimal interval in $F$ and assume by contradiction that $\be-\al>1$.
The set $F'=F \cup \{(\al+1,\be)\}$ contains the facet $F$ properly, therefore it  contains a minimal non-face of $\Delta$.
By \eqref{ConditionIntervals} no  interval in $F$ intersects $(\al, \be)$ properly,  hence the same is true for $(\al+1,\be)$, 
and thus $F'$ contains no  minimal non-face of type \eqref{MonomialP}.
It follows that $F$ contains an interval $(\varepsilon, \zeta)$ such that $\{(\al+1,\be),(\varepsilon, \zeta)\}$
is a minimal non-face of type \eqref{MonomialQ}.
We cannot have $\be <\varepsilon$, otherwise $\{(\al,\be),(\varepsilon, \zeta)\}$ would   be a non-face of type \eqref{MonomialQ} contained in $F$;
we conclude that $\zeta < \al +1$, and in particular that there exists $\be_1 > \zeta$ with $\be=\overline{\be_1}$.

Applying the same argument to the set $F''=F \cup \{(\al,\be-1)\}$,
$F$ contains an interval $(\eta, \nu)$ such that $\{(\al,\be-1),(\eta, \nu)\}$
is a minimal non-face of type \eqref{MonomialQ}. 
We cannot have $\eta>\be-1$, otherwise $\{(\varepsilon, \zeta),(\eta, \nu)\}$ would  be a non-face of type \eqref{MonomialQ} contained in $F$;
as before,  $\nu < \al$ and in fact there exists $\al_1 > \nu$ with $\al= \overline{\al_1}$.

Finally, we derive a contradiction.
If $\zeta \leq \nu$, the face $F$ contains  $\{(\varepsilon,\zeta),(\al,\be)\}$, 
which is the minimal non-face of type \eqref{MonomialQ} arising from the choice of indices $\varepsilon, \zeta, \al_1, \be_1$.
Likewise, if $\zeta > \nu$ then $F$ contains  $\{(\eta,\nu),(\al,\be)\}$, 
which is the minimal non-face of type \eqref{MonomialQ} arising from the choice of indices $\eta, \nu, \al_1, \be_1$.
\end{proof}

\begin{definition}\label{DefinitionEllAlpha}
To each  $\al \in [c-d-2]$ we associate an integer $\ell_\al$ defined as follows.
For each $i=1,\ldots,d$ let $\ga_{\al,i}$  be the  least index $\gamma \geq \al+2$  such that the $\gamma$-th  column of $\mat$ involves variables from $X_i$.
They  are well defined as the $(c-j+1)$-th column involves variables of $X_j$ for each  $j=1,\ldots,d$.
By construction of $\mat$ we have
$$
\{\ga_{\al, 1}, \ldots, \ga_{\al, d}\} = \{ \al+2, \ldots, \al +\ell_\al\}\cup\{c-d+\ell_\al, \ldots, c\}
$$
for an integer   $2 \leq \ell_\al\leq d+1$ that is uniquely determined by $\al$.
The first subset is never empty, whereas the second one is empty if and only if $\ell_\al = d+1$.
\end{definition}

\begin{lemma}\label{LemmaAdmissiblePairsOfUnitaryIntervals}
Let  $\al \in [c-d-2]$ and $\be>\al$.
The pair $\{(\al, \al+1),(\be, \be+1)\}$ is not a non-face of $\Delta$ of type \eqref{MonomialQ} if and only if 
 $$ 
 \beta \in \{ \al+1, \ldots, \al +\ell_\al\}\cup\{c-d+\ell_\al-1, \ldots, c-1\}.
 $$
\end{lemma}
\begin{proof}
The pair $\{(\al, \al+1),(\be, \be+1)\}$ is  a non-face of type \eqref{MonomialQ}  if and only if there exist
$\ga, \de$  such that $\al+1<\ga,\de<\be$, $\og = \be, \od = \be+1$. 
This is equivalent to $\be, \be+1 \ne \ga_{\al,i}$ for every $i = 1, \ldots, d$,
and by Definition \ref{DefinitionEllAlpha} this happens if and only if $\al + \ell_\al < \be $ and $\be+1 < c-d+\ell_\al$,
yielding the desired conclusion.
\end{proof}

The next proposition gives   a complete description of the facets of $\Delta$.

\begin{prop}\label{PropositionDescriptionFacets}
A subset  $F \subset V$ is a facet of $\Delta$ if and only the  if 
 the Hasse diagram $\mathcal{T}$ of the poset $(F,\subseteq)$ satisfies the following conditions:
\begin{itemize}

\item[(i)] $\mathcal{T}$ is a rooted binary tree   with root $(1,c)$;

\item[(ii)] there exists $\al \in [c-d-2]$ such that the leaves of  $\mathcal{T}$ 
are the intervals
$\left\{ (\beta, \beta+1) \, | \, \beta \in \{\al, \ldots, \al + \ell_{\al}\} \cup \{c-d+\ell_{\al}-1, \ldots, c-1\} \right\}$;

\item[(iii)] if  $\fI\in F$ is a node of  $\mathcal{T}$ with one child $\fI_1$, then 
$\length(\fI) = \length(\fI_1) +1$
 and the unique unitary interval in $\fI\setminus \fI_1$ does not belong to $F$;

\item[(iv)] if  $\fI\in F$ is a node of  $\mathcal{T}$ with two children  $\fI_1,  \fI_2$, 
then $\fI_1\cap  \fI_2= \emptyset$ and $ \length(\fI_1) +\length(\fI_2)=\length(\fI) $.
\end{itemize}
In this case we have  $\Card(F) = c+d$.

\end{prop}

\begin{example}\label{ExampleNonGrassmanFacetTree}
Let $\mathbf{n}= (2,2,4,4)$, thus $c=12, d=4$ and 
$$
\mat = \left(
\begin{matrix}
x_{1,0} &x_{2,0} & x_{3,0}  &x_{4,0} &x_{3,1} & x_{4,1}  & x_{3,2}  &x_{4,2} & x_{4,3}  & x_{3,3} & x_{2,1} & x_{1,1}   \\
x_{1,1} &x_{2,1} & x_{3,1}  &x_{4,1} & x_{3,2} & x_{4,2}  & x_{3,3} &x_{4,3} & x_{4,4} & x_{3,4} & x_{2,2} & x_{1,2}
\end{matrix}
\right).
$$
We choose $\al=2$ and determine
$$
\ga_{2,1} = 12, \quad \ga_{2,2}=11, \quad \ga_{2,3}=5, \quad \ga_{2,4}=4,\quad \ell_2=3.
$$
The collection of all the open intervals in the picture below is a facet of $\Delta$.
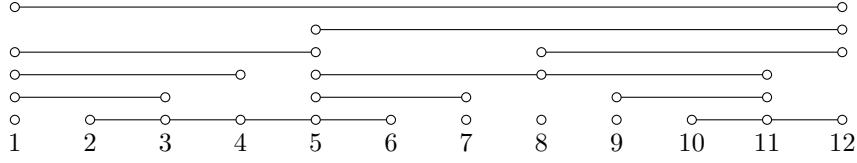
\vspace*{0.3cm}

\begin{center}
\begin{tikzpicture}\label{add}
\draw (1,0.3)--(5,0.3);
\draw (9,0.3)--(11,0.3);

\draw [fill=white] (0,0.3) circle [radius=0.06];
\draw [fill=white] (1,0.3) circle [radius=0.06];
\draw [fill=white] (2,0.3) circle [radius=0.06];
\draw [fill=white] (3,0.3) circle [radius=0.06];
\draw [fill=white] (4,0.3) circle [radius=0.06];
\draw [fill=white] (5,0.3) circle [radius=0.06];
\draw [fill=white] (6,0.3) circle [radius=0.06];
\draw [fill=white] (7,0.3) circle [radius=0.06];
\draw [fill=white] (8,0.3) circle [radius=0.06];
\draw [fill=white] (9,0.3) circle [radius=0.06];
\draw [fill=white] (10,0.3) circle [radius=0.06];
\draw [fill=white] (11,0.3) circle [radius=0.06];

\draw (0,0.6)--(2,0.6);
\draw (4,0.6)--(6,0.6);
\draw (8,0.6)--(10,0.6);

\draw [fill=white] (0,0.6) circle [radius=0.06];
\draw [fill=white] (2,0.6) circle [radius=0.06];
\draw [fill=white] (4,0.6) circle [radius=0.06];
\draw [fill=white] (6,0.6) circle [radius=0.06];
\draw [fill=white] (8,0.6) circle [radius=0.06];
\draw [fill=white] (10,0.6) circle [radius=0.06];

\draw (0,0.9)--(3,0.9);
\draw (4,0.9)--(7,0.9);
\draw (7,0.9)--(10,0.9);

\draw [fill=white] (0,0.9) circle [radius=0.06];
\draw [fill=white] (3,0.9) circle [radius=0.06];
\draw [fill=white] (4,0.9) circle [radius=0.06];
\draw [fill=white] (7,0.9) circle [radius=0.06];
\draw [fill=white] (10,0.9) circle [radius=0.06];

\draw (0,1.2)--(4,1.2);
\draw (7,1.2)--(11,1.2);

\draw [fill=white] (0,1.2) circle [radius=0.06];
\draw [fill=white] (4,1.2) circle [radius=0.06];
\draw [fill=white] (7,1.2) circle [radius=0.06];
\draw [fill=white] (11,1.2) circle [radius=0.06];

\draw (4,1.5)--(11,1.5);
\draw (0,1.8)--(11,1.8);

\draw [fill=white] (4,1.5) circle [radius=0.06];
\draw [fill=white] (11,1.5) circle [radius=0.06];
\draw [fill=white] (0,1.8) circle [radius=0.06];
\draw [fill=white] (11,1.8) circle [radius=0.06];

\node at (0,0) {1};
\node at (1,0) {2};
\node at (2,0) {3};
\node at (3,0) {4};
\node at (4,0) {5};
\node at (5,0) {6};
\node at (6,0) {7};
\node at (7,0) {8};
\node at (8,0) {9};
\node at (9,0) {10};
\node at (10,0) {11};
\node at (11,0) {12};

\end{tikzpicture}
\end{center}
The Hasse diagram of the  poset is

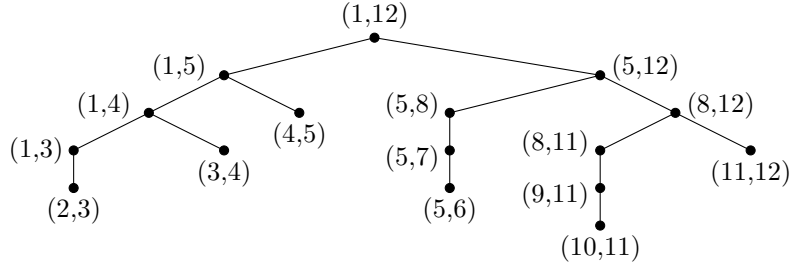
\begin{center}
\begin{tikzpicture}\label{add}

\draw [fill=black] (0,0) circle [radius=0.06];
\node at (0,0.3) {(1,12)};

\draw [fill=black] (-2,-0.5) circle [radius=0.06];
\node at (-2.6,-0.4) {(1,5)};
\draw (-2,-0.5)--(0,0);

\draw [fill=black] (3,-0.5) circle [radius=0.06];
\node at (3.6,-0.4) {(5,12)};
\draw (3,-0.5)--(0,0);

\draw [fill=black] (-1,-1) circle [radius=0.06];
\node at (-1,-1.3) {(4,5)};
\draw (-1,-1)--(-2,-0.5);

\draw [fill=black] (-3,-1) circle [radius=0.06];
\node at (-3.6,-0.9) {(1,4)};
\draw (-3,-1)--(-2,-0.5);

\draw [fill=black] (-4,-1.5) circle [radius=0.06];
\node at (-4.5,-1.5) {(1,3)};
\draw (-4,-1.5)--(-3,-1);

\draw [fill=black] (-2,-1.5) circle [radius=0.06];
\node at (-2,-1.8) {(3,4)};
\draw (-2,-1.5)--(-3,-1);

\draw [fill=black] (-4,-2) circle [radius=0.06];
\node at (-4,-2.3) {(2,3)};
\draw (-4,-2)--(-4,-1.5);

\draw [fill=black] (1,-1) circle [radius=0.06];
\node at (0.5,-0.9) {(5,8)};
\draw (1,-1)--(3,-0.5);

\draw [fill=black] (1,-1.5) circle [radius=0.06];
\node at (0.5,-1.6) {(5,7)};

\draw [fill=black] (1,-2) circle [radius=0.06];
\node at (1,-2.3) {(5,6)};
\draw (1,-2)--(1,-1);

\draw [fill=black] (4,-1) circle [radius=0.06];
\node at (4.6,-0.9) {(8,12)};
\draw (4,-1)--(3,-0.5);

\draw [fill=black] (3,-1.5) circle [radius=0.06];
\node at (2.4,-1.4) {(8,11)};
\draw (4,-1)--(3,-1.5);

\draw [fill=black] (5,-1.5) circle [radius=0.06];
\node at (5,-1.8) {(11,12)};
\draw (4,-1)--(5,-1.5);

\draw [fill=black] (3,-2) circle [radius=0.06];
\node at (2.4,-2.1) {(9,11)};

\draw [fill=black] (3,-2.5) circle [radius=0.06];
\node at (3,-2.8) {(10,11)};
\draw (3,-2.5)--(3,-1.5);

\end{tikzpicture}
\end{center}

\end{example}

\begin{proof}[Proof of Proposition \ref{PropositionDescriptionFacets}]
Assume that $F$ is a facet of $\Delta$.
Then $F$ contains no minimal non-face of type \eqref{MonomialP} or \eqref{MonomialQ}, and it is maximal with respect with this property.

We begin by showing (i).
In order to see that the Hasse diagram $\mathcal{T}$ of $(F,\subseteq)$ is a tree, 
it suffices to observe that for every interval $\fI\in F$ there exists at most one $\fI_1\in F$ with $\fI\subsetneq \fI_1$ and minimal with respect to this property:
if by contradiction there existed two such intervals in $ F$, the facet $F$ would violate the non-crossing condition \eqref{ConditionIntervals}.
Every facet has $(1,c)$ as unique maximal element, since $(1,c)$ does not belong to any non-face, thus $(1,c)$ is the root of $\mathcal{T}$.
Finally, to see that $\mathcal{T}$ is a binary tree, assume by contradiction that it contains a node $(\al,\be)$ with at least three children $(\al_1, \be_1), (\al_2, \be_2),(\al_3, \be_3)$, with
$\al_1 < \be_1 \leq \al_2 < \be_2 \leq\al_3 < \be_3 $.
Let $F'= F\cup\{(\al_1,\be_2)\}$, then $F'$ satisfies \eqref{ConditionIntervals}, so it does not contain a minimal non-face of type \eqref{MonomialP}.
By Corollary \ref{CorollaryReductionToMinimalIntervals}, $F'$ does not contain a minimal non-face of type \eqref{MonomialQ}, contradicting the maximality of the facet $F$.

Next, we prove (ii).
Observe that in general, for any subset $G\subseteq V$ satisfying \eqref{ConditionIntervals}
and any $\beta \in [c-1]$, $G\cup \{(\be, \be+1)\}$ also satisfies \eqref{ConditionIntervals}.
By Lemma \ref{LemmaMinimalIntervalsAreUnitary} the minimal intervals of $F$ are unitary,  let $(\al, \al+1)$ be the one with the least left endpoint $\al$.
We claim that $\al \leq c-d-2$; assume by contradiction that $\al > c-d-2$.
The pair $\{(c-d-2,c-d-1),(\be, \be+1)\}$ is not of type  \eqref{MonomialQ} for any $\be >c-d-2$ by Lemma \ref{LemmaAdmissiblePairsOfUnitaryIntervals};
by Corollary \ref{CorollaryReductionToMinimalIntervals} and Lemma \ref{LemmaMinimalIntervalsAreUnitary}, this implies that $F \cup\{(c-d-2,c-d-1)\}$ is also a face of $\Delta$, contradicting the minimality of $\al$ and yielding  $\al \leq c-d-2$ as claimed.
With the same argument and Lemma \ref{LemmaAdmissiblePairsOfUnitaryIntervals} 
we see that  $(\beta, \beta+1)\in F$ if and only if  $ \beta \in \{\al, \ldots, \al + \ell_{\al}\} \cup \{c-d+\ell_{\al}-1, \ldots, c-1\}$.

Next, we show (iii).
Let   $\fI\in F$ be a node with one child $\fI_1$, and assume there exists an interval $\fI_2\in V$   with $\fI_1\subsetneq \fI_2 \subsetneq \fI$.
Then $F' = F \cup \{ \fI_2\}$ satisfies \eqref{ConditionIntervals} and does not contain a minimal non-face of type \eqref{MonomialQ} by Corollary \ref{CorollaryReductionToMinimalIntervals}, so $F'$ is a face of $\Delta$, contradicting the maximality of the facet $F$. 
Thus such $\fI_2$ cannot exist, and $\length(\fI) = \length(\fI_1) +1$.
By assumption  $\fI_1$ is the only child of $\fI$, so the unique unitary interval in $\fI\setminus \fI_1$ does not belong to $F$.

Finally, we show (iv).
Let $(\al,\be)$ be a node with two children  $(\al_1, \be_1), (\al_2, \be_2)$.
The two intervals are disjoint by \eqref{ConditionIntervals}. 
 Assume by contradiction that $\al<\al_1$, or $ \be_1<\al_2$, or $\be_2 < \be$.
By the same argument used in (i) and (iii), the set $F'$ obtained by adjoining respectively $(\al, \be_1), (\be_1, \be_2)$, or $(\al_2, \be)$ to $F$ is still a face of $\Delta$,  contradiction.

Now we show the opposite direction. 
Assume that $F\subset V$ satisfies properties (i)--(iv), we are going to prove that $F$ is a facet of $\Delta$.

First, we show that $F$ contains no  non-face of type \eqref{MonomialP}, or, equivalently, that $F$ satisfies \eqref{ConditionIntervals}.
Assume by contradiction that $\fI_1 \cap \fI_2 \ne \emptyset$ for two incomparable intervals $\fI_1, \fI_2\in F$.
Since $\mathcal{T}$ is a binary tree there exists a unique smallest interval $\fI\in F$  such that $\fI_1\cup \fI_2 \subseteq \fI$, and this node must have two children 
$\fI_3, \fI_4\in F$ with $\fI_1\subseteq  \fI_3$, $\fI_2\subseteq  \fI_4$.
Clearly $ \fI_1 \cap \fI_2  \subseteq   \fI_3\cap  \fI_4$, so  (iv) yields a contradiction.

Next, we show that $F$ contains no  non-face of type \eqref{MonomialQ}. 
By Corollary \ref{CorollaryReductionToMinimalIntervals} it suffices to show that
$\{\fI_1, \fI_2\}$ is not of type \eqref{MonomialQ} for any two minimal intervals $\fI_1, \fI_2\in F$;
this follows from (ii) and Lemma \ref{LemmaAdmissiblePairsOfUnitaryIntervals}.
Thus $F$ is a face of $\Delta$.

Before proving that  $F$ is a facet, we verify that $\Card(F) = c+d$.
Since $\mathcal{T}$ is a binary tree with  $d+2$ leaves by (ii), 
it has $d+1$ nodes with two children.
By (iii) the nodes with one child correspond to the unitary intervals outside $F$,
 therefore there are $(c-1)-(d+2)$  such nodes.
 The desired formula follows.

Finally, let $F'$ be a facet of $\Delta$ with $F\subseteq F'$.
By the previous paragraph we have $\Card(F)= c+d$; however, 
by the proof of the first direction, we also have $\Card(F')=c+d$, thus $F=F'$ is a facet of $\Delta$. 
\end{proof}

\begin{cor}\label{CorollaryDeltaNonGrassmann}
The complex $\Delta$ is pure  of dimension $c+d-1$.
\end{cor}

Our final  goal in this subsection is to enumerate the facets of $\Delta$.
To this end we consider the following combinatorial  object.
Let $\al,\be_1,\be_2,\ga\in\mathbb{N}$  with $\be_1>0$, $\al+\be_1+\be_2+\ga=c-1$, and such that $\be_2=0$ if $\ga=0$. 
We split the unitary intervals of $V$ in four strings of two different colors: 
the first $\al$ unitary intervals are white, the next  $\be_1$ are black, the next $\ga$ are white, and the last $\be_2$  are black.
$$
 \underbrace{\square  \square \cdots  \square}_\al \underbrace{\blacksquare \blacksquare \cdots  \blacksquare}_{\be_1} \underbrace{ \square  \square \cdots \square}_\ga \underbrace{\blacksquare \blacksquare \cdots \blacksquare}_{\be_2} 
$$
Let $\Sigma(\al,\be_1,\ga,\be_2)$ be the set of all $F\subseteq V$ that satisfy conditions (i), (iii), (iv) of Proposition \ref{PropositionDescriptionFacets} and whose minimal intervals
 are exactly  the black intervals.

We introduce a trivariate generalization of the Catalan numbers known as  the  \emph{Catalan trapezoids} \cite{Re}:
$$
C_{m}(n,k) = 
 \begin{cases} 
{n+k \choose k}  &\mbox{if } 0\leq k < m ,\\ 
{n+k \choose k} - {n+k \choose k - m}  &\mbox{if } m\leq k < n+m-1, \\ 
0  &\mbox{if } k>  n+m-1.
\end{cases}
$$

\begin{lemma}\label{LemmaEnumerationSigmaA}
We have $\Card (\Sigma(\al,\be_1,\ga,\be_2)) =C_{\al+1}(\be_1+\ga+\be_2-1,\al+\be_1+\be_2-1)$.
\end{lemma}
\begin{proof}
We  claim that $\Card (\Sigma(\al,\be_1,\ga,\be_2)) = \chi(\al,\be_1+\be_2,\ga)$ for some
 function $\chi:\mathbb{N}^3\rightarrow\mathbb{N}$  satisfying $\chi(\al,1,0)=\chi(0,1,\ga)=1$ and if $\ga>0,\be>1$
\begin{equation}
\tag{$\ast$}\chi(\al,\be,\ga) = \chi(\al,\ga-1,\be)+\chi(\al,\ga+1,\be-1).
\end{equation}
This is enough to prove the Lemma, 
as the numbers $C_{m}(n,k)$ 
are characterized by the recursion
$
C_m(n,k)=C_m(n-1,k)+C_m(n,k-1)
$
and the boundary conditions $C_m(n,0)=1, C_m(0,k)=1$ for $k\leq m-1$,  $C_m(n,k)=0 $ for $k>n+m-1$, cf. \cite{Re}. 

For any $F\in \Sigma(\al,1,\ga,0)$ the poset $(F,\subseteq)$ is a saturated chain of intervals starting at  $(\al+1,\al+2)$ and ending at $(1,c)$.
For each pair $\fI_1 \subseteq \fI_2$ of consecutive intervals in this chain, $\fI_1$ and $\fI_2$  share one endpoint. 
The left endpoint is shared by two consecutive intervals exactly $\ga$ times, and the right endpoint exactly $\al$ times. 
We get a bijection between $\Sigma(\al,1,\ga,0)$ and the set of lattice paths from $(0,0)$ to $(\al,\ga)$, 
therefore
$
\Card (\Sigma(\al,1,\ga,0)) = {\al+\ga \choose \al}
$.
This verifies the claim when $\be_1=1,\be_2=0$.

Now suppose $\be_1>1$.
The case $\Sigma(0,\be_1,0,0)$ is exactly the Grassmann case, where all the unitary intervals belong to every facet,
therefore $\Card(\Sigma(0,\be_1,0,0))=C_{\be_1-1}$.
Assume $\al>0$, we compute $ \Card (\Sigma(\al,\be_1,0,0))$ using a binary partition.
Let $\fw$ be the last white unitary interval and $\fb$ the first black one:
$$
 \square  \square \cdots \square \overset{\fw}{\square}  \overset{\fb}{\blacksquare} \blacksquare \cdots \blacksquare \blacksquare .
$$
For every $F\in\Sigma(\al,\be_1,0,0)$ exactly one of the following two cases occurs:
\begin{enumerate}
\item
$F \setminus \{\fb\}\in \Sigma(\al+1,\be_1-1,0,0)$. 
Equivalently: every interval in $F$ of length $> 1$ containing $\fb$ also contains some other black unitary interval.
Equivalently: the interval $\fI$ of length 2 that contains $\fb,\fw$ is not in $F$.

\item 
$F \setminus \{\fb\}\notin \Sigma(\al+1,\be_1-1,0,0)$. 
Equivalently: there exists an interval in $F$ of length $> 1$ containing only the black interval $\fb$.
Equivalently: the interval $\fI$ of length 2 that contains $\fb,\fw$ is in $F$.
\end{enumerate}
We define two separate bijections in the two cases.

\begin{enumerate}

\item
We associate $F \mapsto F\setminus \{\fb\} $, i.e. ``we make $\fb$ white''.
This gives a map from case (1) of  $\Sigma(\al,\be_1,0,0)$ to $\Sigma(\al+1,\be_1-1,0,0)$. 
We define a map in the opposite direction by $G \mapsto G \cup \{\fb\}$.
These two maps are inverse to each other, and thus we have a bijection between the two sets.

\item We associate $F \mapsto \tilde{F}$ obtained by throwing $\fI$ away and collapsing $\fw$.
Now this gives a map from case (2) of  $\Sigma(\al,\be_1,0,0)$ to $\Sigma(\al-1,\be_1,0,0)$.
We define a map in the opposite direction by adding a white unitary interval $\fw$ next to $\fb$ and adding the interval $\fI$.
Again, these 2 maps are inverse to each other, and we have a bijection between the two sets.
\end{enumerate}
We deduce that $\Card(\Sigma(\al, \be_1,0,0)) = \Card(\Sigma(\al+1, \be_1-1,0,0))+ \Card(\Sigma(\al-1, \be_1,0,0))$
and by induction this verifies the claim when $\ga=\be_2=0$.

Now suppose that $\be_1>1$ and $\ga>0$;
we compute $\Card(\Sigma(\al,\be_1,\ga,\be_2))$ by means of the same binary partition as in the previous step, but on a different pair of unitary intervals $\fb,\fw$.
Namely, 
we distinguish the rightmost black unitary interval $\fb$ which is adjacent to a white unitary interval $\fw$.
If $\be_2=0$  we have
$$
\square  \square \cdots  \square \blacksquare \blacksquare \cdots \blacksquare  \overset{\fb}{\blacksquare} \overset{\fw}{\square}  \square \cdots \square.
$$
whereas if $\be_2>0$ we have 
$$
\square  \square \cdots  \square \blacksquare \blacksquare \cdots  \blacksquare \square  \square \cdots \square \overset{\fw}{\square}  \overset{\fb}{\blacksquare} \blacksquare \cdots \blacksquare .
$$
For every $F\in \Sigma(\al,\be_1,\ga,\be_2)$ we have   two cases analogous to (1), (2) as above, 
and the same two bijections yield
$
\Card (\Sigma(\al,\be_1,\ga,0)) = \Card (\Sigma(\al,\be_1,\ga-1,0))+ \Card (\Sigma(\al,\be_1-1,\ga+1,0))
$
and 
$\Card (\Sigma(\al,\be_1,\ga,\be_2)) = \Card (\Sigma(\al,\be_1,\ga-1,\be_2))+ \Card (\Sigma(\al,\be_1,\ga+1,\be_2-1))$ if $\be_2>0$.
By induction, the proof of the claim is completed.
\end{proof}

\begin{cor}\label{CorollaryEnumerationFacetsNonGrassmann}
The number of facets of  $\Delta$ is 
$$
f_{c+d}(\Delta) = \sum_{\al=1}^{c-d-2}{c+d-1\choose \al+d}-(c-d-2){c+d-1 \choose d}.
$$
\end{cor}
\begin{proof}
Denoting by $\Sigma_\al$ the set of facets of $\Delta$ whose leftmost unitary interval is $(\al,\al+1)$, by definition we have
  $\Sigma_\al = \Sigma(\al-1, \ell_\al+1,c-\al-d-2,d-\ell_\al+1) $. 
By Proposition \ref{PropositionDescriptionFacets} (ii) and    Lemma \ref{LemmaEnumerationSigmaA}  we obtain 
\begin{align*}
f_{c+d}(\Delta) &=  \sum_{\al=1}^{c-d-2}\Card(\Sigma_\al) = \sum_{\al=1}^{c-d-2}\left( {c+d-1\choose \al+d}-{c+d-1 \choose d}\right)
\\
& = \sum_{\al=1}^{c-d-2}{c+d-1\choose \al+d}-(c-d-2){c+d-1 \choose d}.
\end{align*}
\end{proof}

\subsection{Defining equations of $\FIg$}

We  combine the results on $\Delta$  with those from \cite{BCV,CHV} to  prove the main theorem  of this section.

\begin{thm}\label{TheoremGrobnerFiber}
The defining ideal $\mK$ of the special fiber ring of a rational normal scroll  is minimally generated by the polynomials 
$\PP$ with  $(\al,  \be, \ga, \de)\in {[c] \choose 4}$ and 
$\QQ$ with $(\al,  \be, \ga, \de)\in {[c-d] \choose 4}$,
and they form a squarefree quadratic Gr\"obner basis with respect to the term order $\prec$. 
\end{thm}
\begin{proof}
Since $\PP,\QQ\in \mK$  the inclusion  $ \mI \subseteq \iin_\prec(\mK)$ holds.
The ideal $\mI$ is clearly squarefree. 
Moreover, it is unmixed as $\Delta$ is a pure simplicial complex by Proposition \ref{PropositionDeltaGrassmann} and Corollary \ref{CorollaryDeltaNonGrassmann};
it follows from the associativity formula for multiplicities that  $ \mI = \iin_\prec(\mK)$ if and only if the  factor rings
$\SF/\mI$ and $ \SF/ \iin_\prec(\mK)$ have the same Krull dimension and multiplicity.
In other words, it suffices to show that $\dim (\SF/\mI) = \dim (\FIg) $ and  $e(\SF/\mI) = e(\FIg)$.

These invariants  are determined  for the Stanley-Reisner ring $\SF/\mI$ in 
Proposition \ref{PropositionDeltaGrassmann}  and Corollaries \ref{CorollaryDeltaNonGrassmann}, \ref{CorollaryEnumerationFacetsNonGrassmann},
and they only depend on $c,d$.
They agree with those of $\FIg$ when the scroll $\mathcal{S}_{n_1,\ldots,n_d}$ is balanced \cite[Section 4]{CHV}.
However, by \cite[Theorem 3.7]{BCV}  the Hilbert function and hence the multiplicity of the special fiber ring of a  scroll depend only on $c,d$, 
and thus we get the equality  $ \mI = \iin_\prec(\mK)$ for any rational normal scroll.

We conclude that $\{P_\ba,Q_\bb\,|\, \ba\in {[c]\choose 4}, \bb\in {[c-d]\choose 4}\}$ is a Gr\"obner basis of $\mK$, and in particular a (minimal) set of generators.
\end{proof}

Recall that a standard graded $\mathbb{K}$-algebra $A$ is   \emph{Koszul} if $\Tor_i^A(\mathbb{K},\mathbb{K})_j = 0 $  whenever $i\ne j$,
and that this condition holds whenever $A$ is presented by a Gr\"obner basis of quadrics.
We refer to \cite{Co} for details.

\begin{cor}\label{CorollarySpecialFiberKoszul}
The special fiber ring of a rational normal scroll is Koszul.
\end{cor}

\begin{remark}\label{RemarkInitialGrassmann}
In the special case when the ideal $\mK$ is only generated by the Pl\"ucker relations, i.e. when $c<d+4$, 
$\FIg$ is the coordinate ring of the Grassmann variety $\mathbb{G}(1,c-1) \subseteq \mathbb{P}^{ {c \choose 2} -1}$.
Our initial complex $\Delta$ is different from the classical one in the theory of straightening laws \cite{MS,SW}:
in that context the leading monomial of $\PP$ is $T_{\al,\de} T_{\be,\ga}$ and 
the resulting initial complex is a \emph{non-nesting} complex,
cf. \cite{PPS,SSW} for related considerations.
The initial complex $\Delta'$ constructed in \cite{CHV} for balanced scrolls  also follows the classical choice, and in fact it is different from our $\Delta$ for all values $c\geq 4, d\geq 1$.
For example, it can be seen that $\Delta'$ has 4 cone points if $c < d+4$, 2 if $c = d+4$, and none if $c > d+4$, 
whereas $\Delta$ has  $c$ cone points if $c<d+4$, $c-2 $ if $c=d+4$, and at least one if $c \geq d+4$.
\end{remark}

\section{The defining equations of the Rees ring}\label{SectionRees}

The first result of this section gives the defining equations of the Rees ring of a rational normal scroll.
It is proved in \cite[Theorem 3.7]{BCV} that the ideal $I$ is of 	\emph{fiber type},
that is, 
$\mJ$ is generated in bidegrees $(\ast,1)$ and $(0,\ast)$.
In other words, $\RI$  is defined by the equations of $\FIg$ and of the symmetric algebra of $I$.
Combining this  with Theorem \ref{TheoremGrobnerFiber} we  obtain:

\begin{thm}\label{TheoremGeneratorsRees}
The defining ideal $\mJ$ of the Rees ring of $I$ is minimally generated by the polynomials 
$\PP$ with  $(\al,  \be, \ga, \de)\in {[c] \choose 4}$, 
$\QQ$ with  $(\al,  \be, \ga, \de)\in {[c-d] \choose 4}$, and $\LL, \MM$ with 
$(\al,  \be, \ga)\in {[c] \choose 3}$.
\end{thm}

We are going to show that these polynomials form a Gr\"obner basis of $\mJ$.
Unlike the case of $\mK$, we already know  they generate $\mJ$ 
so we can accomplish this goal via Buchberger's Criterion.
A priori there are ${4+1 \choose 2 } =10$ types of S-pairs to examine,
however, thanks to Proposition \ref{PropositionBypassSpairs},
 we  will only need to produce explicit equations of S-pair reduction for the 3 types $\S(Q,M)$, $\S(Q,L)$, $\S(L,M)$.

We begin with an analysis of the syzygies  of $\mJ$.
This is
motivated by the following observation, which supplies a practical  method to prove the reduction of  an S-pair to 0.
Denote by $\Supp(\cdot)$ the set of monomials appearing in a polynomial.

\begin{remark}\label{RemarkReductionSpair}
Suppose we have a polynomial ring equipped with a  term order, and an equation 
$$
m_1F_1+ \sum_{i= 2}^M m_i F_i = n_1 G_1 + \sum_{j=2}^N n_j G_j
$$
where $m_i, n_j$ are monomials, $F_i,G_j$ are polynomials,
and $\lcm\big(\LM(F_1),\LM(G_1)\big) = m_1\LM(F_1) = n_1\LM(G_1)$.
If the following conditions hold

\begin{itemize}

\item the highest two monomials in $\cup_i \Supp(m_iF_i)$ appear only in $m_1 F_1$;

\item the highest monomial in $\cup_j \Supp(n_jG_j)$ appears only in $n_1 G_1$;

\item the second highest monomial in $\cup_j \Supp(n_jG_j)$ appears only in $n_2G_2$;
\end{itemize}
then  $\S(F_1,G_1)$ reduces to 0 modulo the set $\{F_i,G_j\}$.
In fact, we have $$
\S(F_1,G_1) = \sum_{j=2}^N n_j G_j -\sum_{i=2}^M m_i F_i 
$$
with $\LM(\S(F_1,G_1))\geq \LM(n_j G_j),\LM(m_i F_i)$ for $i,j\geq 2$.
\end{remark}

We adopt a more flexible notation, allowing arbitrary tuples as subscripts of  equations.
Denote by $\mathfrak{S}_n$ the symmetric group on $[n]$ and by $\sigma(\cdot)$ the sign of  a permutation.

\begin{notation}\label{NotationPermutations}
We extend Notation \ref{NotationIndicesVariables} to
 equations  \eqref{equationL},  \eqref{equationM}, \eqref{equationP} by setting
 $$
 L_{\al_{i_1},\al_{i_2},\al_{i_3}} := {\sigma( \bi)}  L_{\al_{1},\al_{2},\al_{3}},
 \quad 
 M_{\al_{i_1},\al_{i_2},\al_{i_3}} := {\sigma (\bi)}  M_{\al_{1},\al_{2},\al_{3}},
$$
$$
 P_{\be_{j_1},\be_{j_2},\be_{j_3}, \be_{j_4}} := {\sigma (\bj)}  P_{\be_{1},\be_{2},\be_{3}, \be_{4}}
$$
for all $(\al_{1},\al_{2},\al_{3} )\in {[c]\choose 3}, (\be_{1},\be_{2},\be_{3}, \be_{4})\in {[c]\choose 4}$,
 $({i_1},{i_2},{i_3})\in \mathfrak{S}_3, ({j_1},{j_2},{j_3}, {j_4}) \in \mathfrak{S}_4$,
$$
 L_{\al_{1},\al_{2},\al_{3}} := 0,
 \quad 
 M_{\al_{1},\al_{2},\al_{3}} := 0,
 \quad 
 P_{\be_{1},\be_{2},\be_{3}, \be_{4}} := 0
$$
 whenever 
 $\Card\{{\al_1},{\al_2},{\al_3}\}<3, \Card\{{\be_1},{\be_2},{\be_3}, {\be_4}\}<4.$
\end{notation}

\begin{lemma}
For every $(\al_1,\al_2,\al_3,\al_4,\al_5)\in {[c-d]\choose 5}$
there are two syzygies
\begin{equation}\label{S1}\tag{S1}
 \sum{\sigma(\bi)} T_{\overline{\al_{i_1}},\overline{\al_{i_2}}}M_{\al_{i_3},\al_{i_4},\al_{i_5}}=
\sum{\sigma(\bj)} \mu_{2,\al_{j_1}} Q_{\al_{j_2},\al_{j_3},\al_{j_4},\al_{j_5}}
\end{equation}
\begin{equation}\label{S2}\tag{S2}
 \sum {\sigma(\bi)} T_{{\al_{i_1}},{\al_{i_2}}}M_{\overline{\al_{i_3}},\overline{\al_{i_4}},\overline{\al_{i_5}}}=
\sum {\sigma(\bj)}\mu_{2,\overline{\al_{j_1}}} Q_{\al_{j_2},\al_{j_3},\al_{j_4},\al_{j_5}}
 \end{equation}
where the sums range over all $\bi=(i_1, i_2, i_3, i_4, i_5), \bj=(j_1, j_2, j_3, j_4, j_5) \in \mathfrak{S}_5$
 with $i_1<i_2, i_3< i_4< i_5$ and $j_2< j_3 < j_4< j_5$.

For every
$(\al_1,\al_2,\al_3,\al_4)\in {[c-d]\choose 4}$ and $ \ep\in[c]$
there are three syzygies
\begin{equation}
\label{S3}\tag{S3}
\sum {\sigma(\bi)}T_{{\al_{i_1}},{\al_{i_2}}}L_{\overline{\al_{i_3}},\overline{\al_{i_4}},\ep}=
Q_{\al_1,\al_2,\al_3,\al_4}\mu_{1,\ep}
+ \sum  {\sigma(\bj)} M_{\al_{j_1},\al_{j_2},{\al_{j_3}}} T_{\overline{\al_{j_4}},\ep}
\end{equation}
\begin{equation}
\label{S4}\tag{S4}
  \sum {\sigma(\bi)} T_{\overline{\al_{i_1}},\overline{\al_{i_2}}}M_{{\al_{i_3}},{\al_{i_4}},\ep} =
Q_{\al_1,\al_2,\al_3,\al_4}\mu_{2,\ep}
  +\sum {\sigma(\bj)}L_{\overline{\al_{j_1}},\overline{\al_{j_2}},\overline{\al_{j_3}}} T_{{\al_{j_4}},\ep}
\end{equation}
\begin{align*}
\label{S5}\tag{S5}
&\sum {\sigma(\bi)}T_{\overline{\al_{i_1}},\overline{\al_{i_2}}}L_{{\al_{i_3}},{\al_{i_4}},\ep}
 +  \sum {\sigma(\bh)} M_{{\al_{h_1}},{\al_{h_2}},\ep} T_{\al_{h_3},\overline{\al_{h_4}}}
\\
& \!- \sum \sigma(\bj)  P_{\al_{j_1},\al_{j_2},\al_{j_3},\ep}\mu_{2,\overline{\al_{j_4}}}+ \sum \sigma(\bj) M_{\al_{j_1},{\al_{j_2}},{\al_{j_3}}} T_{\overline{\al_{j_4}},\ep}\\
&\!=\! \mu_{1,\ep}Q_{\al_1,\al_2,\al_3,\al_4}  \!+\!  \sum  {\sigma(\bh)}  L_{\overline{\al_{h_1}},\overline{\al_{h_2}},\al_{h_3}}T_{\al_{h_4},\ep}
\!+\!\sum {\sigma(\bh)} T_{\al_{h_1},\al_{h_2}}  M_{{\al_{h_3}},\overline{\al_{h_4}},\ep}
\end{align*}
where  the sums range over all permutations
$\bi = (i_1,i_2,i_3,i_4), \bj= (j_1, j_2, j_3,j_4),\bh= (h_1, h_2, h_3,h_4)\in \mathfrak{S}_4$
with $i_1<i_2, i_3<i_4, j_1<j_2<j_3, h_1<h_2$.

For every $(\al_1,\al_2,\al_3,\al_4)\in {[c]\choose 4}$ there is a syzygy
\begin{equation}\label{S6}\tag{S6}
\sum \sigma(\bi)  L_{\al_{i_1},\al_{i_2},\al_{i_3}} \mu_{2,{\al_{i_4}}}
=-\sum \sigma(\bi) M_{\al_{i_1},\al_{i_2},\al_{i_3}}  \mu_{1,{\al_{i_4}}}
\end{equation}
where the sums range over  all permutations $\bi = (i_1, i_2, i_3, i_4) \in \mathfrak{S}_4$ with $i_1<i_2<i_3$.
\end{lemma}

\begin{proof}
Both sides of  \eqref{S1} are equal to   $\sum{\sigma(\bh)}\mu_{2,\al_{h_1}} T_{{\al_{h_2}},{\al_{h_3}}} T_{\overline{\al_{h_4}},\overline{\al_{h_5}}}$ 
where the sum ranges over all  $\bh=(h_1, h_2, h_3, h_4, h_5)\in\mathfrak{S}_5$ with $h_2<h_3$, $h_4<h_5$.
Likewise, both sides of \eqref{S2} are equal to  $\sum{\sigma(\bh)}\mu_{2,\overline{\al_{h_1}}} T_{{\al_{h_2}},{\al_{h_3}}} T_{\overline{\al_{h_4}},\overline{\al_{h_5}}}$.

Since  $\mu_{1, \overline{\al_i}}=\mu_{2,\al_i}$,
both sides of \eqref{S3} are equal  to
$$
\sum \sigma(i_1,i_2,i_3,i_4) T_{{\al_{i_1}},{\al_{i_2}}}T_{\overline{\al_{i_3}},\overline{\al_{i_4}}}\mu_{1,\ep} - 
\sum \sigma(j_1,j_2,j_3,j_4) T_{{\al_{j_1}},{\al_{j_2}}}T_{\overline{\al_{j_3}},\ep}\mu_{1,\overline{\al_{j_4}}} 
$$
with sums ranging over all permutations with $i_1<i_2, i_3<i_4, j_1<j_2$.
Likewise, both sides of \eqref{S4} are equal to 
$$
\sum \sigma(i_1,i_2,i_3,i_4) T_{\overline{\al_{i_1}},\overline{\al_{i_2}}}T_{{\al_{i_3}},{\al_{i_4}}}\mu_{2,\ep} -
\sum \sigma(j_1,j_2,j_3,j_4) T_{\overline{\al_{j_1}},\overline{\al_{j_2}}}T_{{\al_{j_3}},\ep}\mu_{2,{\al_{j_4}}}.
$$
Both sides of \eqref{S5} are equal to
\begin{align*}
&\sum\sigma(\bi)T_{\overline{\al_{i_1}},\overline{\al_{i_2}}}T_{{\al_{i_3}},{\al_{i_4}}}\mu_{1,\ep}
-
\sum\sigma(\bj) T_{\overline{\al_{j_1}},\overline{\al_{j_2}}}T_{\al_{j_3},\ep}\mu_{1,{\al_{j_4}}}+
\sum \sigma(\bj)T_{{\al_{j_1}},{\al_{j_2}}}T_{{\al_{j_3}},\overline{\al_{j_4}}}\mu_{2,\ep}
-\\
&\sum \sigma(\bh) T_{{\al_{h_1}},\overline{\al_{h_2}}}T_{\al_{h_3},\ep}\mu_{2,{\al_{h_4}}}-
 \sum \sigma(\bj)T_{\al_{j_1},\al_{j_2}}T_{\al_{j_3},\ep}\mu_{2,\overline{\al_{j_4}}} -
\sum \sigma(\bj) T_{\al_{j_1},\al_{j_2}} T_{\overline{\al_{j_3}},\ep} \mu_{2,{\al_{j_4}}}
\end{align*}
with sums ranging over all permutations with  $i_1<i_2, i_3< i_4, j_1<j_2$.

Finally, both sides of \eqref{S6} equal $\sum \sigma(\bi) \mu_{1,{\al_{i_1}}}\mu_{2,\al_{i_2}} T_{\al_{i_3},\al_{i_4}}$
with sum over  all permutations $\bi$ such that $i_3<i_4$.
\end{proof}

\begin{lemma}\label{LemmaOrderBars}
Let $\al, \be, \ga \in [c-d]$. 
If $\oa < \ob < \og$ then  $\be > \min(\al,\ga)$.
\end{lemma}
\begin{proof}
Note that $\be \ne \al, \ga$.
Suppose that $\be<\al,\ga$.
By construction of $\mat$,
  $\ob>\oa$ forces $\ob>c-d$,
 but we cannot have $\be < \ga $ and $c-d< \ob < \og$, contradiction.
\end{proof}

\begin{prop}\label{PropositionSpairQM}
The S-pairs of type $\S(Q,M)$ reduce to 0.
\end{prop}
\begin{proof}
Given a linear S-pair $\S(\QQ,M_{\ep,\zeta,\eta})$ with $(\al,\be,	\ga,\de)\in{[c-d]\choose 4}, (\ep,\zeta,\eta)\in {[c]\choose 3}$,
by Proposition \ref{PropLeadingTerms} there are two cases.

If $\{\ep,\eta\} = \{\al,\be\}$ then  $\S(M_{\ep,\zeta,\eta},\QQ) = T_{\og,\od}M_{\al,\zeta,\be} - \mu_{2,\zeta}\QQ$.
We consider  the syzygy \eqref{S1} for indices  $\al < \zeta < \be < \ga < \de\leq c-d$.
By  Definition \ref{DefinitionTermOrder},
the highest three monomials on either side of \eqref{S1} are those whose $\mdeg$ is equal to $\ee_{\al+d} + \ee_{\zeta+d}+ \ee_{\be+d}+ \ee_{\og+d}+ \ee_{\od+d}$,
specifically the three monomials
$$\mu_{2,\zeta}T_{\al,\be}T_{\og,\od} \,\succ \,\mu_{2,\be}T_{\al, \zeta}T_{\og,\od}\, \succ \,\mu_{2,\al}T_{\zeta,\be}T_{\og,\od}.$$
On the left-hand side of \eqref{S1}, they appear only in $T_{\og,\od} M_{\al, \zeta, \be}$, whereas on the right-hand side 
$\mu_{2,\zeta}T_{\al,\be}T_{\og,\od} $ appears only in $\mu_{2,\zeta} \QQ$ and 
$\mu_{2,\be}T_{\al, \zeta}T_{\og,\od}$ appears only in $\mu_{2,\be}Q_{\al,\zeta,\ga,\de}$.
By Remark \ref{RemarkReductionSpair}, the S-pair reduces to 0.

If  $\{\ep,\eta\}=\{\og,\od\}$ then   $\S(M_{\ep,\zeta,\eta},\QQ) = T_{\al,\be}M_{\ep, \zeta,\eta} - \mu_{2,\zeta}\QQ$.
Note that we must have   $\zeta = \overline{\theta}$ for some $\theta\ne \de$, 
and by Lemma \ref{LemmaOrderBars} we  have $\theta > \gamma$.
We consider the syzygy \eqref{S2} for indices $\al <\be<\ga < \theta, \de \leq c-d$.
The highest  monomials on either side of \eqref{S2} are those whose $\mdeg$ is equal to $\ee_{\al+d} + \ee_{\be+d}+ \ee_{\og+d}+ \ee_{\od+d}+ \ee_{\overline{\theta}+d}$,
specifically the three monomials
$$\mu_{2,\overline{\theta}}T_{\al,\be}T_{\og,\od} \,\succ \,\mu_{2,\od}T_{\al, \be}T_{\og,\overline{\theta}}\, \succ \,\mu_{2,\og}T_{\zeta,\be}T_{\overline{\theta},\od}.$$
On the left-hand side of \eqref{S2}, they appear only in $T_{\al,\be} M_{\ep, \zeta, \eta}$, whereas on the right-hand side 
$\mu_{2,\overline{\theta}}T_{\al,\be}T_{\og,\od} $ appears only in $\mu_{2,\zeta} \QQ$ and 
$\mu_{2,\od}T_{\al, \be}T_{\og,\overline{\theta}}$ appears only in $\mu_{2,\od}Q_{\al,\be,\ga,\theta}$.
By Remark \ref{RemarkReductionSpair}, the S-pair reduces to 0.
\end{proof}

\begin{prop}\label{PropositionSpairQL}
The S-pairs of type $\S(Q,L)$ reduce to 0.
\end{prop}
\begin{proof}
Let  $\S(Q_{\al_1,\al_2,\al_3,\al_4},L_{\be,\ga,\de})$ be a linear S-pair
with $(\al_1,\al_2,\al_3,\al_4)\in {[c-d]\choose 4},$ $ ({\be,\ga,\de})\in {[c] \choose 3}$,
and let $\ep$ be such that $\mu_{1,\ep}$ divides $\LM(L_{\be,\ga,\de})$,
 so that $\LM(L_{\be,\ga,\de}) = \mu_{1,\ep}T_{\al_1,\al_2}$ or $\LM(L_{\be,\ga,\de}) = \mu_{1,\ep}T_{\oa_3,\oa_4}$. 

Assume first $\LM(L_{\be,\ga,\de}) = \mu_{1,\ep}T_{\al_1,\al_2}$.
Then $\mu_{1,\ep} \succ \mu_{1,\al_1},\mu_{1,\al_2}$ by Proposition \ref{PropLeadingTerms},
and hence $\mdeg(\mu_{1,\ep}) > \mdeg(\mu_{1,\al_1}) > \mdeg(\mu_{2,\al_1}) = \ee_{\al_1+d}$ by Remark \ref{RemarkOrderVariablesMultidegree}.
We consider the syzygy \eqref{S5}.
Comparing $\mdeg$, the highest monomials in either side are those divisible by $\mu_{1,\ep}$.
In particular, the highest monomial is $\mu_{1,\ep}T_{\al_1,\al_2} T_{\oa_3,\oa_4}$, appearing only in $L_{\al_1,\al_2,\ep}T_{\oa_3,\oa_4}$ in the left-hand side and only in $\mu_{1,\ep}Q_{\al_1,\al_2,\al_3,\al_4}$ in the right-hand side; 
the second highest monomial is $\mu_{1,\ep}T_{\al_1,\al_3} T_{\oa_2,\oa_4}$, appearing only in $L_{\al_1,\al_3,\ep}T_{\oa_2,\oa_4}$ in the left-hand side and only in $\mu_{1,\ep}Q_{\al_1,\al_2,\al_3,\al_4}$ in the right-hand side.
By Remark \ref{RemarkReductionSpair}, the S-pair $\S(Q_{\al_1,\al_2,\al_3,\al_4},L_{\be,\ga,\de})$ reduces to 0.

Now assume  $\LM(L_{\be,\ga,\de}) = \mu_{1,\ep}T_{\oa_3,\oa_4}$.
Then $\mu_{1,\ep} \succ \mu_{1,\oa_3},\mu_{1,\oa_4}$ by Proposition \ref{PropLeadingTerms}.
Therefore   $\mdeg(\mu_{1,\ep}) > \mdeg(\mu_{1,\oa_3}) = \mdeg(\mu_{2,\al_3}) = \ee_{\al_3+d}$ and 
 $\mdeg(\mu_{1,\ep}) > \mdeg(\mu_{1,\oa_4}) = \mdeg(\mu_{2,\al_4}) = \ee_{\al_4+d}$
  by Remark \ref{RemarkOrderVariablesMultidegree}.
Observe that we also have $\mdeg(\mu_{1,\ep}) > \mdeg(\mu_{2,\ep}) = \ee_{\ep+d}$.
We consider the syzygy \eqref{S3}, and we are going to use Remark \ref{RemarkReductionSpair}.
Since $\mdeg(\mu_{1,\ep}) >\ee_{\ep+d},  \ee_{\al_3+d},  \ee_{\al_4+d}$, 
the highest monomial in either side of \eqref{S3} is the only one whose $\mdeg$ is equal to
$\mdeg(\mu_{1,\ep})+\ee_{\al_1+d}+\ee_{\al_2+d}+\ee_{\oa_3+d}+\ee_{\oa_4+d}$, 
that is,  $\mu_{1,\ep}T_{\al_1,\al_2}T_{\oa_3,\oa_4}$.
It appears only in $T_{\al_1,\al_2}L_{\be,\ga,\de}$ in the left-hand side, and only in $Q_{\al_1,\al_2,\al_3,\al_4}\mu_{1,\ep}$ in the right-hand side. 

If $\mdeg(\mu_{1,\ep}) < \ee_{\al_2+d}$,  the second highest monomial in either side of \eqref{S3} is the only one  whose $\mdeg$ is equal to
$\ee_{\al_1+d}+\ee_{\al_2+d}+\ee_{\ep+d}+\ee_{\al_3+d}+\ee_{\oa_4+d}$, 
that is,  $T_{\al_1,\al_2}T_{\ep,\oa_4}\mu_{1,\oa_3}$.
It appears only in $T_{\al_1,\al_2}L_{\be,\ga,\de}$ in the left-hand side, and only in $M_{\al_1,\al_2,\al_3}T_{\oa_4,\ep}$ in the right-hand side. 

If $\mdeg(\mu_{1,\ep}) > \ee_{\al_2+d}$,  the second highest monomial in either side of \eqref{S3} is the only one  whose $\mdeg$ is equal to
$\ee_{\al_1+d}+\mdeg(\mu_{1,\ep}) +\ee_{\al_3+d}+\ee_{\oa_2+d}+\ee_{\oa_4+d}$, 
that is,  $T_{\al_1,\al_3}T_{\oa_2,\oa_4}\mu_{1,\ep}$.
It appears only in $T_{\al_1,\al_3}L_{\oa_2,\oa_4,\ep}$ in the left-hand side, and only in $Q_{\al_1,\al_2,\al_3,\al_4}\mu_{1,\ep}$ in the right-hand side. 

Finally, if $\mdeg(\mu_{1,\ep}) = \ee_{\al_2+d}$, then we have $\ep = \oa_2$ by Definition \ref{DefinitionTermOrder}.
In this case, the monomials in the previous two paragraphs have the same $\mdeg$, and also the same $\sdeg$.
However, we have $T_{\al_1,\al_3}T_{\oa_2,\oa_4}\mu_{1,\ep} \succ T_{\al_1,\al_2}T_{\ep,\oa_4}\mu_{1,\oa_3}$,
because $\mu_{1,\ep}$ is the largest of the six variables involved.
Thus, $T_{\al_1,\al_3}T_{\oa_2,\oa_4}\mu_{1,\ep}$ is the second highest monomial in either side of \eqref{S3}.

Using Remark \ref{PropLeadingTerms}, we conclude in all cases  that the S-pair $\S(Q_{\al_1,\al_2,\al_3,\al_4},L_{\be,\ga,\de})$ reduces to 0,
and the proof is completed.
\end{proof}

\begin{prop}\label{PropositionSpairLM}
The S-pairs of type $\S(L,M)$  reduce to 0.
\end{prop}
\begin{proof}
An S-pair $\S(\LL,M_{\de,\zeta,\ep})$,  where  $(\al,\be,\ga),(\de,\zeta,\ep)\in{[c] \choose 3}$, 
must have bidegree equal to $(1,2)$ or $(2,1)$.

Suppose the S-pair has bidegree $(1,2)$,
then  $ \gcd \big(\LM(\LL),\LM(M_{\de,\zeta,\ep})\big)=\mu_{2,\zeta} $, and
by Remark \ref{RemarkLeadingL} either $\overline{\zeta} = \al$ or $\overline{\zeta} = \ga$.
If   $\overline{\zeta} = \al$ then $\mu_{2,\zeta}=\mu_{1,\al}\succ\mu_{1,\be},\mu_{1,\ga}$. 
It follows by Proposition \ref{PropLeadingTerms} that $ \mu_{1,\be},\mu_{1,\ga}$ cannot be of the form $x_{i,0}$ for any $i$, equivalently,
there exist $\iota, \theta$ such that $\be = \overline{\theta}, \ga = \overline{\iota}$;
similarly if  $\overline{\zeta} = \ga$.
In any case we get $\{\al,\be,\ga\} = \{\overline{\zeta},  \overline{\iota},\overline{\theta}\}$ for some $\iota, \theta\leq c-d$.
Observe that $\zeta< \theta, \iota$ since  $ \mu_{2, \zeta}\succ  \mu_{2,	\iota},\mu_{2,\theta}$.
We consider the syzygy \eqref{S4} for indices  $\de < \zeta < \theta, \iota \leq c-d$ and $\ep$.
The highest monomials in either side are those whose $\mdeg$ is equal to 
$\ee_{\ep+d}+\ee_{\de+d}+\ee_{\zeta+d}+\ee_{\overline{\theta}+d}+\ee_{\overline{\iota}+d}$,
and they are $T_{\de,\ep}T_{\overline{\theta},\overline{\iota}}\mu_{2,\zeta} \succ T_{\de,\zeta}T_{\overline{\theta},\overline{\iota}}\mu_{2,\ep} \succ T_{\zeta,\ep}T_{\overline{\theta},\overline{\iota}}\mu_{2,\de}$.
On the left-hand side of \eqref{S4} they appear only in $ T_{\overline{\theta},\overline{\iota}} M_{\de,\zeta,\ep}$,
whereas on the right-hand side $T_{\de,\ep}T_{\overline{\theta},\overline{\iota}}\mu_{2,\zeta} $ appears only in $T_{\de,\ep}\LL$ and 
$ T_{\de,\zeta}T_{\overline{\theta},\overline{\iota}}\mu_{2,\ep}$ appears only in $Q_{\de,\zeta,\theta,\iota}\mu_{2,\ep}$.
By Remark \ref{RemarkReductionSpair}, the S-pair reduces to 0.

Now suppose the S-pair  has bidegree $(2,1)$.
If $\LM(\LL) = \mu_{1,\ga} T_{\al,\be}$ then  $\{\de,\ep\}=\{\al,\be\}$ and $\mu_{1,\ga}\succ \mu_{1,\al},\mu_{1,\be}$ by Proposition \ref{PropLeadingTerms}.
Since  $\al = \de <\zeta<\ep = \be$ it follows   by Remark \ref{RemarkFirstRow} that  
 $ \mu_{1,\zeta} \prec \mu_{1, \al}$ or $\mu_{1,\zeta} \prec \mu_{1, \be}$; 
 in either case we conclude 
$\mu_{1,\ga}\succ \mu_{1,\zeta}$.
Therefore $\mdeg(\mu_{1,\ga}) > \mdeg(\mu_{1,\al}), \mdeg(\mu_{1,\be}), \mdeg(\mu_{1,\zeta})$ by Remark \ref{RemarkOrderVariablesMultidegree}. 
We consider the syzygy \eqref{S6} for indices $\al < \zeta< \be<\ga$.
In either side, the highest monomials are
those divisible by $\mu_{1,\ga}$, namely
$\mu_{1,\ga} \mu_{2,\zeta} T_{\al,\be} \succ \mu_{1,\ga} \mu_{2,\be} T_{\al,\zeta} \succ\mu_{1,\ga} \mu_{2,\al} T_{\zeta,\be}$.
In the right-hand side, they appear only in $\mu_{1,\ga} M_{\de,\zeta,\ep}$, whereas in the left-hand side  
 $\mu_{1,\ga} \mu_{2,\zeta} T_{\al,\be}$ appears only in $\mu_{2,\zeta}\LL$ and 
 $ \mu_{1,\ga} \mu_{2,\be} T_{\al,\zeta}$ appears only in $\mu_{2,\be}L_{\al,\zeta,\ga}$.
By Remark \ref{RemarkReductionSpair}, the S-pair reduces to 0.
 If $\LM(\LL) = \mu_{1,\al} T_{\be,\ga}$ then  $\{\de,\ep\}=\{\be, \ga\}$ and $\mu_{1,\al}\succ \mu_{1,\be},\mu_{1,\ga}$, and we
proceed in the same way  using the syzygy  \eqref{S6} for indices  $\al< \be< \zeta< \ga$.
\end{proof}

The following proposition circumvents the explicit reduction of the other S-pairs.
It is based on the  observation that most S-pairs involve minors from at most 5 columns; 
this allows to  identify several small Gr\"obner bases in $\mJ$, 
and to reduce the desired statement to a Hilbert series equation to be checked in a finite number of cases.
Note that the hypothesis on the ground field is not restrictive for our aim, 
as we will prove Theorem \ref{TheoremGrobnerRees} for any $\mathbb{K}$.

\begin{prop}\label{PropositionBypassSpairs}
Assume $\ch(\mathbb{K})=0$.
The S-pairs of types $\S(P,P)$, $\S(P,Q)$, $\S(Q,Q)$, $\S(P,M)$, $\S(P,L)$, $\S(M,M)$, $\S(L,L)$ reduce to 0.
\end{prop}

\begin{proof}
The fact that  S-pairs of types $\S(P,P)$, $\S(P,Q)$, $\S(Q,Q)$ reduce to 0 follows from Theorem \ref{TheoremGrobnerFiber} by Buchberger's Criterion.
For any non-trivial S-pair  of the form
$\S(\PP,M_{\ep,\zeta,\eta})$,
$ \S(\PP,L_{\ep,\zeta,\eta})$, or
$ \S(\MM,M_{\de,\ep,\zeta})$, 
$ \S(\LL,L_{\de,\ep,\zeta})$,
we  have $\Card\{\al,\be,\ga,\de,\ep,\zeta,\eta\}\in\{4,5\}$,
respectively, $\Card\{\al,\be,\ga,\de,\ep,\zeta\}\in\{4,5\}$.
By Buchberger's Criterion, the statement of the proposition is implied by the following stronger claim:
for any  $\Gamma \subseteq [c]$ with $\Card(\Gamma)\in\{4,5\}$,
 the two  subsets of $\mJ$
$$
\mathcal{G}_{M,\Gamma}= \left\{P_\ba,M_\bb\,|\,\ba\in{\Gamma \choose 4},\bb\in{\Gamma \choose 3}\right\}, \quad
\mathcal{G}_{L,\Gamma}=\left\{P_\ba,L_\bb\,|\,\ba\in{\Gamma \choose 4},\bb\in{\Gamma \choose 3}\right\}
$$
are Gr\"obner bases.

In order to prove the claim, let $m= \Card(\Gamma)\in\{4,5\}$.
After renaming the variables that appear in  $\mathcal{G}_{M,\Gamma}$,  we can rewrite the set $\mathcal{G}_{M,\Gamma}$ as
\begin{align*}
\mathcal{G}_{m}=& \quad \left\{ T_{\al,\be}T_{\ga,\de} - T_{\al,\ga}T_{\be,\de} + T_{\al,\de}T_{\be,\ga}\, \Big|\, (\al,\be,\ga,\de)\in{[m] \choose 4} \right \}\\
& \cup \left\{   x_\al T_{\be,\ga}-x_\be T_{\al,\ga}+x_\ga T_{\al,\be} \, \Big|\, (\al,\be,\ga)\in{[m] \choose 3} \right \}
\end{align*}
and by Proposition \ref{PropLeadingTerms} the set of leading monomials of the elements of $\mathcal{G}_{M,\Gamma}$ becomes
\begin{align*}
\mathcal{M}_{m}=& \quad \left\{  T_{\al,\ga}T_{\be,\de} \, \Big|\, (\al,\be,\ga,\de)\in{[m] \choose 4} \right \} \cup \left\{ x_\be T_{\al,\ga} \, \Big|\, (\al,\be,\ga)\in{[m] \choose 3} \right \}.
\end{align*}
Thus $\Card(\mathcal{G}_{m})=\Card(\mathcal{M}_{m})=5$  if $m= 4 $,  while $\Card(\mathcal{G}_{m})=\Card(\mathcal{M}_{m})=15$ if $m=5$.
The statement  that  $\mathcal{G}_{M,\Gamma}$ is a Gr\"obner basis amounts to the equation of Hilbert series
\begin{equation}\label{EquationHilbertSeriesM}
\HS\left( \frac{A}{(\mathcal{G}_{m})}\right) =  \HS\left( \frac{A}{(\mathcal{M}_{m})}\right) 
\end{equation}
where $A = \mathbb{K}[x_1, \ldots, x_m, T_{1,2},T_{1,3},\ldots,T_{m-1,m}]$, and $m\in \{4,5\}$.

Now we consider $\mathcal{G}_{L,\Gamma}$.
After renaming the variables that appear in  $\mathcal{G}_{L,\Gamma}$,  we can rewrite the set $\mathcal{G}_{L,\Gamma}$ also as $\mathcal{G}_{m}$ above.
On the other hand,
by Proposition \ref{PropLeadingTerms}, the set of leading monomials of the elements of $\mathcal{G}_{L,\Gamma}$ becomes
\begin{align*}
\mathcal{L}_{m}=& \quad \left\{  T_{\al,\ga}T_{\be,\de} \, \Big|\, (\al,\be,\ga,\de)\in{[m] \choose 4} \right \}\\
 &\cup \left\{ x_{\al_1} T_{\al_2,\al_3} \, \Big|\, 	\al_i\in [m],\, \al_1\ne\al_2\ne\al_3\ne\al_1, \mbox{ and }   x_{\al_1}\succ x_{\al_2}, x_{\al_3} \right \}.
\end{align*}
In other words, the leading monomial of each $x_\al T_{\be,\ga}-x_\be T_{\al,\ga}+x_\ga T_{\al,\be}$ is the monomial  with largest $x$ variable  with respect to $\prec$.
By Remark \ref{RemarkFirstRow} the  order $\prec$ on  $x_1, \ldots, x_m$ is such that,
for any $(\al,\be,\ga)\in {[m] \choose 3}$,
if $x_\be\succ x_\ga$ then $x_\al\succ x_\be$.
There are 8 such total orders for $m=4$ 
\begin{align*}
 x_1\! \succ\! x_2 \!\succ\! x_3\! \succ\! x_4,\quad &&  x_1\! \succ\! x_2 \!\succ\! x_4\! \succ\! x_3,\quad 
&& x_1\! \succ\! x_4\! \succ\! x_2\! \succ\! x_3,\quad && x_1\! \succ\! x_4\! \succ\! x_3\! \succ\! x_2,\\
  x_4 \!\succ\! x_1\! \succ\! x_2\! \succ\! x_3,\quad &&  x_4 \!\succ\! x_1\! \succ\! x_3\! \succ\! x_2,\quad
&&  x_4 \!\succ\! x_3\! \succ\! x_1\! \succ\! x_2,\quad &&  x_4 \!\succ\! x_3\! \succ\! x_2\! \succ\! x_1, 
\end{align*}
and 16  such total orders for $m=5$
\begin{align*}
&& x_1\! \succ\! x_2 \!\succ\! x_3\! \succ\! x_4\! \succ\! x_5,\quad && x_1\! \succ\! x_2 \!\succ\! x_3\! \succ\! x_5\! \succ\! x_4,\quad 
&& x_1\! \succ\! x_2 \!\succ\! x_5\! \succ\! x_3\! \succ\! x_4,\quad\\
&& x_1\! \succ\! x_2 \!\succ\! x_5\! \succ\! x_4\! \succ\! x_3,\quad && x_1\! \succ\! x_5 \!\succ\! x_2\! \succ\! x_3\! \succ\! x_4,\quad 
&& x_1\! \succ\! x_5 \!\succ\! x_2\! \succ\! x_4\! \succ\! x_3,\quad\\
&& x_1\! \succ\! x_ 5\!\succ\! x_4\! \succ\! x_2\! \succ\! x_3,\quad && x_1\! \succ\! x_5\!\succ\! x_4\! \succ\! x_3\! \succ\! x_2,\quad 
&& x_5\! \succ\! x_1 \!\succ\! x_2\! \succ\! x_3\! \succ\! x_4,\quad\\
&& x_5\! \succ\! x_1 \!\succ\! x_2\! \succ\! x_4\! \succ\! x_3,\quad && x_5\! \succ\! x_1 \!\succ\! x_4\! \succ\! x_2\! \succ\! x_3,\quad 
&& x_5\! \succ\! x_1 \!\succ\! x_4\! \succ\! x_3\! \succ\! x_2,\quad\\
&& x_5\! \succ\! x_4 \!\succ\! x_1\! \succ\! x_2\! \succ\! x_3,\quad && x_5\! \succ\! x_4 \!\succ\! x_1\! \succ\! x_3\! \succ\! x_2,\quad 
&& x_5\! \succ\! x_4 \!\succ\! x_3\! \succ\! x_1\! \succ\! x_2,\quad\\
&& x_5\! \succ\! x_4 \!\succ\! x_3\! \succ\! x_2\! \succ\! x_1. \quad&&  
&& 
\end{align*}
As before, 
we have $\Card(\mathcal{L}_{m})=5$  if $m= 4 $ and $\Card(\mathcal{L}_{m})=15$ if $m=5$, 
and  
the statement  that  $\mathcal{G}_{L,\Gamma}$ is a Gr\"obner basis is equivalent  to the equation of Hilbert series
\begin{equation}\label{EquationHilbertSeriesL}
\HS\left( \frac{A}{(\mathcal{G}_{m})}\right) =  \HS\left( \frac{A}{(\mathcal{L}_{m})}\right). 
\end{equation}

With the aid of the computer algebra system \texttt{Macaulay2} \cite{M2},
when $\mathbb{K}=\mathbb{Q}$,
we have verified that  the  $1+1+8+16 = 34$ monomial ideals $(\mathcal{M}_m)$ and $(\mathcal{L}_m)$ have the same Hilbert series as $(\mathcal{G}_m)$ for $m=4,5$;
in other words, 
we have verified  the validity of equations  \eqref{EquationHilbertSeriesM} and \eqref{EquationHilbertSeriesL}  in all  $34$ instances. 
 By flatness, \eqref{EquationHilbertSeriesM} and \eqref{EquationHilbertSeriesL}  hold for any field extension $\mathbb{Q}\subseteq \mathbb{K}$.
\end{proof}

We remark that it is also possible to prove Proposition \ref{PropositionBypassSpairs} with similar techniques as Propositions \ref{PropositionSpairQM}, \ref{PropositionSpairQL}, \ref{PropositionSpairLM}.

We  combine the results obtained so far to  prove the main theorem  of this section.

\begin{thm}\label{TheoremGrobnerRees}
The minimal generators of the defining ideal $\mJ$ of the Rees ring of a rational normal scroll
from a squarefree quadratic Gr\"obner basis with respect to the term order $\prec$.
\end{thm}
\begin{proof}
Observe that all the leading monomials are squarefree.
Assume first that  $\ch(\mathbb{K})=0$. 
Then the conclusion follows from Buchberger's Criterion, since all the S-pairs  reduce to 0.
Equivalently,
the Hilbert series of the ideal generated by the leading monomials of the generators equals the one of $\mJ$.
However, the Hilbert series of a monomial ideal does not depend on $\mathbb{K}$,
while it follows from \cite[Theorem 3.7]{BCV} that the Hilbert series of $\mJ$ does not depend on  $\mathbb{K}$ either.
Thus the statement holds for any ground field.
\end{proof}

\begin{cor}
The Rees ring of a rational normal scroll is a Koszul algebra.
\end{cor}

\begin{remark}
Let $\mathcal{S}_{n_1, \ldots, n_d}\subseteq \mathbb{P}^{c+d-1}$ be a rational normal scroll and  $I=(\mathbf{g})$ its homogeneous ideal.
The Main Theorem implies that,
expressing  $\RI$ and $\FIg$ as quotients of $\SR=\mathbb{K}[x_1, \ldots,x_{c+d}, T_1, \ldots, T_{c-1\choose 2}]$ and $\SF=\mathbb{K}[T_1, \ldots, T_{c-1\choose 2}]$ respectively, 
their first graded Betti numbers 
$$
\dim_\mathbb{K} \Tor_1^{\SR}(\RI,\mathbb{K})_j
\quad \text{and} \quad
\dim_\mathbb{K} \Tor_1^{\SF}(\FIg,\mathbb{K})_j
$$
depend only on the dimension and codimension of the scroll, 
but not on the  partition $n_1, \ldots, n_d$.
It would be interesting to know whether the same is true for the higher Betti numbers of the blowup algebras of $\mathcal{S}_{n_1, \ldots, n_d}\subseteq \mathbb{P}^{c+d-1}$.
\end{remark}

\begin{example}[The Veronese surface]\label{ExampleVeronese}
It is interesting to study the implicitization problem for all nondegenerate projective varieties $V \subseteq \mathbb{P}^N$ of minimal degree.
By the classification of Del Pezzo and Bertini  \cite{EH},
$V$ is a cone over a smooth such variety, 
and smooth varieties of minimal degree are precisely the rational normal scrolls, the smooth quadric hypersurfaces, and the Veronese surface $\nu_2(\mathbb{P}^2)\subseteq \mathbb{P}^5$.

We have shown that the Rees ring and the special fiber  of any rational normal scroll are Koszul algebras presented by squarefree Gr\"obner bases of quadrics.
The same result holds for quadric hypersurfaces, trivially:
both blowup algebras are polynomial rings.
The homogeneous  ideal  $I\subseteq \mathbb{K}[x_0, \ldots, x_5]$  of the Veronese surface is generated by the 2-minors of the generic symmetric matrix
$$
\left(
\begin{matrix}
x_{0} & x_{1}  &  x_{2} \\
x_{1} & x_{3}  &  x_{4} \\
x_{2} & x_{4}  &  x_{5} 
\end{matrix}
\right).
$$
It is an ideal of linear type, so the special fiber is a polynomial ring, 
and $\RI$ is isomorphic to the symmetric algebra of $I$ and thus it is defined by quadrics.
However, it turns out that  $\RI$ is not a Koszul algebra:
if  $\mathbb{K}$ has characteristic 0 then we have
$ \dim_\mathbb{K}\Tor_6^{\RI}(\mathbb{K},\mathbb{K})_7 = 32 $.
This confirms the fact that the Veronese surface exhibits an exceptional behavior in various contexts,
see for instance \cite[Section  3.4]{Ru}, \cite[Example 3.5]{Co}.
It also gives the first known example of a prime ideal with linear powers \cite{BCV} whose Rees ring is not Koszul.
\end{example}

\subsection*{Acknowledgments}
This work was supported by a fellowship of the Purdue Research Foundation, 
while the author was visiting the University of Genoa in 2016, 
and by Grant No. 1440140 of the National Science Foundation, while he was a Postdoctoral Fellow at the Mathematical Sciences Research Institute in Berkeley in 2018.

The author would like to thank the Department of Mathematics in Genoa for the warm hospitality, and  especially Aldo Conca for several valuable discussions.
The author also thanks the referee for  carefully reading  the manuscript and for helpful comments that improved its clarity.

\end{document}